\documentclass[a4paper, 11pt]{amsart}

\usepackage[dvipsnames]{xcolor}

\usepackage{iftex}
\ifPDFTeX 
  \usepackage[utf8]{inputenc}
  \usepackage[T1]{fontenc}
\else 
  \usepackage{fontspec}
\fi
\usepackage{lmodern}

\usepackage{geometry}
\usepackage[textwidth=28mm]{todonotes}
\usepackage{mathtools}
\usepackage{yfonts}
\usepackage[foot]{amsaddr}

\usepackage[pdfencoding=auto]{hyperref}
\hypersetup{
    colorlinks,
    citecolor=green,
    filecolor=black,
    linkcolor=blue,
    urlcolor=black
}

\usepackage[english]{babel}

\usepackage{dsfont}
\usepackage{comment}
\usepackage{enumerate}
\usepackage[shortlabels]{enumitem}
\usepackage{tikz}
\usepackage{pgfplots}
\pgfplotsset{compat=1.15}
\usetikzlibrary{arrows}

\numberwithin{equation}{section}

\usepackage{tabularx}

\newcommand{\nocontentsline}[3]{}
\let\origcontentsline\addcontentsline
\newcommand\stoptoc{\let\addcontentsline\nocontentsline}
\newcommand\resumetoc{\let\addcontentsline\origcontentsline}

\theoremstyle{plain}
\newtheorem{thm}{Theorem}[section]

\newtheorem{cor}[thm]{Corollary} 
\newtheorem{lem}[thm]{Lemma} 
\newtheorem{rem}[thm]{Remark} 

\newtheorem{prop}[thm]{Proposition} 
\theoremstyle{definition} 
 
\theoremstyle{remark}

\newcommand{\TODO}[1]{{\color{red}{#1}}}
\newcommand{\innerproduct}[2]{\langle #1, #2 \rangle}

\newcommand{\R}{\mathbb{R}}

\newcommand{\x}{x}

\newcommand{\uv}{v}
\newcommand{\f}{{f}}

\newcommand{\eps}{\varepsilon}

\definecolor{codegreen}{rgb}{0,0.6,0}
\definecolor{codegray}{rgb}{0.5,0.5,0.5}
\definecolor{codepurple}{rgb}{0.58,0,0.82}
\definecolor{backcolour}{rgb}{0.95,0.95,0.92}

\author{Thomas Rey}
\author{Tommaso Tenna}
\address[Thomas Rey]{Université Côte d’Azur, CNRS, LJAD, Parc Valrose, F-06108 Nice, France}
\address[Tommaso Tenna]{Université Côte d’Azur, CNRS, LJAD, Parc Valrose, F-06108 Nice, France}
\title{The Boltzmann equation for a multi-species inelastic mixture}
\date{}

\begin{document}

\begin{abstract}
    A granular gas is a collection of macroscopic particles that interact through energy-dissipating collisions, also known as inelastic collisions.  This inelasticity is characterized by a collision mechanics in which mass and momentum are conserved and kinetic energy is dissipated.  Such a system can be described by a kinetic equation of the Boltzmann type. Nevertheless, due to the macroscopic aspect of the particles, any realistic description of a granular gas should be written as a mixture model composed of $M$ different species, each with its own mass. We propose in this work such a granular multi-species model and analyse it, providing Povzner-type inequalities, and a Cauchy theory in general Orlicz spaces. We also analyse its large time behavior, showing that it exhibits a mixture analogue of the seminal Haff's Law.

\end{abstract}


\maketitle

\vspace{-1cm}
\tableofcontents

\section{Introduction}

Boltzmann type kinetic equations arise in many important applications as reliable models for the dynamics of gases which are not in thermodynamical equilibrium. Classical kinetic theory describes the behavior of dilute gas through a very large number of particles interacting via microscopic elastic collisions, see \cite{cercignani1988, cercignani1994} for a detailed discussion. On the contrary, several applications requires to model phenomena with low-density gases of dissipatively colliding particles, such as the study of spaceship re-entry in a dusty atmosphere or the description of planetary rings. In this framework, particles are ``microscopic" with respect to the scale of the system, but ``macroscopic" in the modeling, since they could represent grains of a given size, for instance a suspension of pollen or sand in a fluid. From a thermodynamical point of view, granular gases are open systems, since dissipative collisions imply a persistent loss of the mechanical energy of the grains with a consequent decay of the temperature. We refer to the book by Brilliantov and P\"oschel \cite{brilliantov2004} for a complete introduction to the subject.

The most important feature of particle interactions in granular flows is their inelastic behavior: the total kinetic energy is not preserved in the collisions. Therefore, to account for dissipative collisions, the Boltzmann equation has to be generalized. The first study of inelastic hard spheres-like model started with the seminal paper \cite{haff1983} and it has drawn attention for the unexpected physical behavior of the solutions to this system. Indeed, it involves different phenomena, as kinetic \textit{collapse} at the kinetic level and \textit{clustering} at hydrodynamic level \cite{goldhirsch1993}. The first mathematical contribution to granular gases is probably the paper by Benedetto, Caglioti and Pulvirenti in \cite{benedetto1997}, followed by many other works \cite{alonso2009existence, bobylevcarrillogamba2000, bobylev2004, mischlermouhot2006II, mischlermouhot2006}. The rigorous study of the hydrodynamic limit for the inelastic Boltzmann system is still not complete, with few rigorous results \cite{jabin2017, alonsolodstristani2020, alonsolodstristani2022}. For a review on inelastic Boltzmann equation, we suggest the review by Villani \cite{villani2006} and the more recent one \cite{carrillo2020recent}.\\
An accurate description of physical problems involves the presence of particles having different masses, especially when dealing with grains. In the elastic framework, the first derivation of the Boltzmann equation for gas mixtures was proposed by Chapman and Cowling in \cite{chapman1939} and more recently, the analytical setting has been studied in \cite{briant2016stability, briant2016boltzmann}. The analytical study of the inelastic framework is still incomplete. Most of the literature is focused on the hydrodynamics property of inelastic mixtures, see for instance \cite{garzo2002, garzo2007, serero2006, serero2015, serero2011}.

The aim of this work is to develop a theory for the Boltzmann equation of inelastic hard sphere mixtures.  
The rest of the paper is organized as follows. Section \ref{Section_Multispecies} is devoted to the presentation of the main properties of the model, describing both the microscopic interactions and the macroscopic behavior of the solutions. The functional framework is presented in Section \ref{Section_Homogeneous}, with the entropy inequality and the study of the kinetic energy dissipation. In Section \ref{Section_Povzner_inequality} we extend the result for the Povzner-type inequality of \cite{gamba2004} to the multi-species case. The existence and uniqueness of the solutions are proven in Section \ref{Section_MischlerMouhot}, inspired by the work of Mischler, Mouhot and Ricard \cite{mischlermouhot2006}. Finally, in Section \ref{Section_Haffs_Law} we give a characterization of the long-time behavior of the solutions, proving a generalized version of the Haff's law based upon the use of entropy, in the spirit of \cite{alonso2013}.

\section{The Multi-species Granular Gases Equation}
\label{Section_Multispecies} 
Consider a mixture of $M$ macroscopic species, each described by a distribution function $f_i(t,x,v)$, where the mass of a particle of component $i$ is $m_i >0$. The multi-species Boltzmann equation for granular gases in absence of external forces reads as:
\begin{equation}
	\label{granular_gases}
	\begin{cases}
		\partial_t f_i + \uv \cdot \nabla_\x f_i = \displaystyle \frac{1}{\eps} \mathcal{Q}^{\, inel \,}_i(\f), \qquad \f=(f_1,\dots,f_M),\\
		\mathcal{Q}^{inel}_i (\f) = \displaystyle \sum_{j=1}^N \mathcal{Q}_{i,j} (f_i,f_j),\\ 
		f_i(0,\x,\uv)= f_i^0 (\x,\uv).
	\end{cases}
\end{equation}
In this framework $\mathcal{Q}^{\, inel \,}_i(\f)$ is the inelastic collision operator, which describes the dissipative microscopic collision dynamics. This energy dissipation might be due to a non-perfect restitution of the energy during collisions or to the roughness of the surface, but it does not affect the conservation of momentum. The parameter $\eps >0$ is the \textit{Knudsen number}, which is the ratio between the mean free path of particles before a collision and the length scale of observation. In other terms, it governs the frequency of collisions: a small Knudsen number describes more frequent collisions.\\
The collision operator $\mathcal{Q}^{\, inel \,}_i$ is given by a sum of different \textit{collision operators} $\mathcal{Q}_{i, j}$, referred to as self-collision operators if $i=j$ (interactions between particles of the same species) or cross-collision operators if $i \neq j$ (interactions between particles of different species).

\subsection{Microscopic dynamics}
\label{Subsection_Micro_Dynamics}
The microscopic dynamics is fundamental to understand the formulation of the Boltzmann equation for granular gases and to derive an explicit formulation for \eqref{granular_gases}. Indeed, the dissipation of kinetic energy implies important differences with respect to the classical Boltzmann equation. We need the following reasonable assumptions:
\begin{enumerate}
    \item The particles interact via \textit{binary} collisions. This means that collisions between $3$ or more particles can be neglected.
    \item These binary collisions are localized in space and time.
    \item Collisions preserve mass for each species, total mass and total momentum, dissipating a fraction $1-e$ of the total kinetic energy, where $e \in [0,1]$ is the so called \textit{restitution coefficient}:
    \begin{equation}
    	\label{microscopic_dynamics}
        \begin{cases}
            m_i v' + m_j v'_* = m_i v + m_j v_*,\\
            \displaystyle m_i |v'|^2 + m_j |v'_*|^2 - m_i |v|^2 - m_j |v_*|^2 = -\frac{m_i\,m_j}{(m_i+m_j)}\frac{1-e^2}{2} (1-\innerproduct{\nu}{\sigma})|v-v_*|^2 \leq 0,
        \end{cases}
    \end{equation}
    with $\sigma \in \mathbb{S}^{d-1}$ and $\nu:=\displaystyle \frac{v-v_*}{|v-v_*|}$. 
\end{enumerate}
\begin{rem}
	It is possible to consider also a non-constant restitution coefficient, usually depending on the relative velocity $|v-v_*|$ of the colliding particles. More precisely, it should be chosen close to the elastic case $e=1$ for small relative velocities (when the dissipation is negligible) and decaying towards zero for large relative velocities.
\end{rem}
The conservation of total momentum and the dissipation of total energy yield the following parametrizations for the post-collisional velocities $\left((v^{ij})', (v^{ij}_*)'\right)$ in terms of the pre-collisional ones $(v,v_*)$:
\begin{itemize}
\item The $\omega$-representation
\begin{equation}
\label{omega_representation}
    \begin{aligned}
        (v^{ij})'=v - \frac{m_j}{m_i+m_j} (1+e) ((v-v_*) \cdot \omega)\, \omega,\\
        \displaystyle (v^{ij}_*)'=v_* + \frac{m_i}{m_i+m_j} (1+e) ((v-v_*) \cdot \omega)\, \omega,\\
    \end{aligned} 
\end{equation}
where $\omega \in \mathbb{S}^{d-1}$.
\item The $\sigma$-representation
\begin{equation}
\label{sigma_representation}
    \begin{aligned}
        (v^{ij})'=\frac{m_i v+ m_j v_*}{m_i+m_j} + \frac{2m_j}{m_i+m_j} \frac{(1-e)}{4} (v-v_*) + \frac{2m_j}{m_i+m_j} \frac{1+e}{4} |v-v_*| \sigma,\\
        (v^{ij}_*)'=\frac{m_i v+ m_j v_*}{m_i+m_j} - \frac{2m_i}{m_i+m_j} \frac{(1-e)}{4} (v-v_*) - \frac{2m_i}{m_i+m_j} \frac{1+e}{4} |v-v_*| \sigma,\\
    \end{aligned} 
\end{equation}
where $\sigma \in \mathbb{S}^{d-1}$, which is related to $\omega$ according to
\begin{equation}
    (g \cdot \omega)\, \omega = \frac{1}{2} ( g- |g| \sigma).
\end{equation}
\end{itemize}
We would like to underline that the restitution coefficient $e=1$ corresponds to the case of the classical multi-species elastic collision dynamics \cite{briant2016boltzmann}.\\
To derive the strong form of the Boltzmann equation it is necessary to consider the expressions of the pre-collisional velocities $('(v^{ij}),\, '(v^{ij}_*))$, in terms of the post-collisional ones $(v, v_*)$:
\begin{equation}
\label{pre_collisional}
    \begin{aligned}
        '(v^{ij})=v - \frac{m_j}{m_i+m_j} \frac{1+e}{e} ((v-v_*) \cdot \omega)\,\omega,\\
        \displaystyle '(v^{ij}_*)=v_* + \frac{m_i}{m_i+m_j} \frac{1+e}{e} ((v-v_*) \cdot \omega)\,\omega,\\
    \end{aligned} 
\end{equation}
In this case $'(v^{ij})$ and $'(v^{ij}_*)$ do not coincide with $(v^{ij})'$ and $(v^{ij}_*)'$. This is due to the irreversibility of the collisions, in accordance to the dissipative behavior of granular gases. It implies that the transformation $(v^{ij},v_*^{ij},\sigma) \to ((v^{ij})',(v_*^{ij})',\sigma)$ is not an involution, differently from the elastic case.\\
Henceforth, we will avoid the superscripts $v^{ij}$ in the notations and we will simply refer to the pre-collisional velocities as $(v, v_*)$ and to the post-collisional ones as $(v', v_*')$. \\

\noindent \textbf{Remark.} It is important to observe that this model is meaningful also in dimension $1$. This is not the case of elastic collisions, for which one-dimensional collisions become trivial. Indeed, such collisions are only $\{v',v'_*\}=\{v,v_*\}$, which in the classical Boltzmann equation should result only in a change of direction. However, since particles are indistinguishable, the collision dynamics are exactly the same if particle velocities are swapped or preserved. In particular in $1$D the collision operator $\mathcal{Q}_i (f) \equiv 0$.\\ 
On the contrary, for the inelastic case we have
\begin{equation*}
    \{v',v'_*\}=\{v,v_*\} \quad \text{ or } \quad \left \{ \frac{v+v_*}{2} \pm \frac{e}{2} \frac{m_k}{m_i+m_j} (v-v_*) \right \}
\end{equation*}
depending on the value of $\sigma \in \{-1, +1 \}$, where $k \in \{i,j\}$.

\subsection{Geometry of inelastic collisions for mixtures}
We have pointed out that the conservation of total momentum and the dissipation of energy are consequences of the relations \eqref{sigma_representation}. Let us rewrite the expressions of the post-collisional velocities as
\begin{equation}
\begin{aligned}
    &v'=\Omega_- + \frac{1+e}{2}\frac{m_j}{m_i+m_j}|v-v_*|\sigma, \\ 
    &v_*'=\Omega_+ - \frac{1+e}{2}\frac{m_i}{m_i+m_j}|v-v_*|\sigma,
\end{aligned} 
\end{equation}
defining $\Omega_+$ and $\Omega_-$ as the new center of mass of the particles $i$ and $j$ respectively:
\begin{equation}
\begin{aligned}
    \Omega_- :=\frac{m_i}{m_i+m_j}v+\frac{m_j}{m_i+m_j}v_* + \frac{2m_j}{m_i+m_j} \frac{(1-e)}{4} (v-v_*),\\
    \Omega_+ :=\frac{m_i}{m_i+m_j}v+\frac{m_j}{m_i+m_j}v_* - \frac{2m_i}{m_i+m_j} \frac{(1-e)}{4} (v-v_*).
\end{aligned}
\end{equation}
Using this expression of $\Omega_\pm$ we can also rewrite the pre-collisional velocities as 
\begin{equation}
\begin{aligned}
    v=\Omega_- +\frac{m_j}{m_i+m_j}(v-v_*)-\frac{2m_j}{m_i+m_j} \frac{(1-e)}{4} (v-v_*), \\ v_*= \Omega_+ -\frac{m_i}{m_i+m_j}(v-v_*)+\frac{2m_i}{m_i+m_j} \frac{(1-e)}{4} (v-v_*).
 \end{aligned}
\end{equation}
Therefore, in Fig. \ref{geometrical_configuration} it is shown a possible configuration of the pre- and post-collisional velocities in $2$D in the phase space.
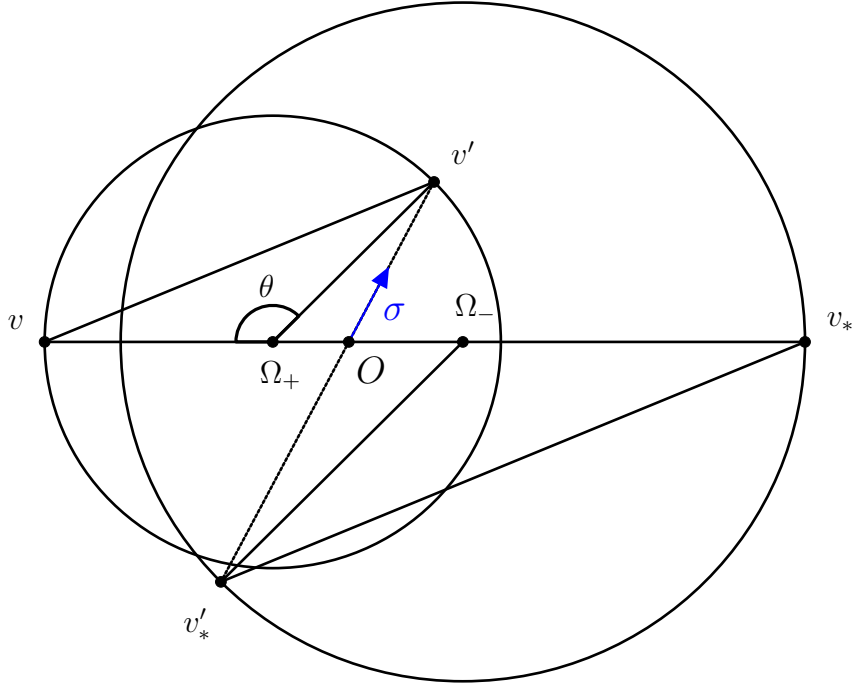
\begin{figure}[h!]
\begin{center}          
    \begin{tikzpicture}[line cap=round,line join=round,>=triangle 45,x=1cm,y=1cm]
    	\draw [shift={(-2,0)},line width=1.2pt] (0,0) -- (45:0.4841163593948704) arc (45:180:0.4841163593948704) -- cycle;
    	\draw [line width=1pt] (-2,0) circle (3cm);
    	\draw [line width=1pt] (0.5,0) circle (4.5cm);
    	\draw [line width=1pt,dash pattern=on 1pt off 1pt] (0.12132034355964283,2.121320343559643)-- (-2.6819805153394642,-3.1819805153394642);
    	\draw [line width=1pt] (-5,0)-- (5,0);
    	\draw [line width=1pt] (0.12132034355964283,2.121320343559643)-- (-5,0);
    	\draw [line width=1pt] (-2.6819805153394642,-3.1819805153394642)-- (5,0);
    	\draw [line width=1pt] (0.5,0)-- (-2.6819805153394642,-3.1819805153394642);
    	\draw [line width=1pt] (-2,0)-- (0.12132034355964283,2.121320343559643);
    	\draw [->,line width=0.8pt,color=blue] (-1,0) -- (-0.47,1);
    	\begin{large}
    		\draw [fill=black] (-5,0) circle (2pt);
    		\draw[color=black] (-5.377021082128359,0.267) node {$v$};
    		\draw [fill=black] (5,0) circle (2pt);
    		\draw[color=black] (5.470362096437763,0.267) node {$v_*$};
    		\draw [fill=black] (-2,0) circle (2pt);
    		\draw[color=black] (-1.8945600226063182,-0.43866972808872496) node {$\Omega_+$};
    		\draw [fill=black] (0.5,0) circle (2pt);
    		\draw[color=black] (0.668079954065469,0.4511710628973322) node {$\Omega_-$};
    		\draw [fill=black] (0.12132034355964283,2.121320343559643) circle (2pt);
    		\draw[color=black] (0.5049926845193162,2.508645882189719) node {$v'$};
    		\draw [fill=black] (-2.6819805153394642,-3.1819805153394642) circle (2pt);
    		\draw[color=black] (-2.9922334671842637,-3.740191156336913) node {$v_*'$};
    		\draw [fill=black] (-1,0) circle (2pt);
    		\draw[color=black] (-0.7084749420888858,-0.37226990318511783) node {\Large $O$};
    		\draw[color=black] (-2.081853110122214,0.7390931750640916) node {$\theta$};
    		\draw[color=blue] (-0.4,0.4) node {\Large $\mathbf{\sigma}$};
    	\end{large}
    \end{tikzpicture}
\end{center}
\caption{Geometrical configuration of the inelastic collision between two particles with different masses in the phase space.} \label{geometrical_configuration}
\end{figure}
We observe that the points $v$ and $v_*$ (describing the velocity for particle $i$) belong to the plane defined by the center of mass $\Omega_+$ and $\text{Span}(\sigma, v-v_*)$ and, similarly, $v'$ and $v'_*$ (describing the velocity for particle $j$) belong to the plane defined by the center of mass $\Omega_-$ and $\text{Span}(\sigma, v-v_*)$. Moreover, we can notice that
\begin{equation}
    \frac{m_i+m_j}{m_j}|\Omega_- - v'|=\frac{m_i+m_j}{m_i}|\Omega_+-v'_*|=\frac{1+e}{2}|v-v_*|,
\end{equation}
namely we have
\begin{equation*}
    v' \in \mathcal{B}\left(\Omega_-, \frac{m_j}{m_i+m_j}\frac{1+e}{2}|v-v_*| \right) \qquad \text{ and } \qquad v'_* \in \mathcal{B}\left(\Omega_+, \frac{m_i}{m_i+m_j}\frac{1+e}{2}|v-v_*| \right),
\end{equation*}
where $\mathcal{B}(x,r)$ is the $d$-ball centered in $x$ of radius $r$, as shown in Fig. \ref{geometrical_configuration}.

\subsection{Weak form}
\label{Subsection_Weak_Form} 

\subsubsection{Mono-species Collision Operator}
The operators $\mathcal{Q}_{i,i}^{inel}$ correspond to the classical Boltzmann collision operators for granular gases, since they describe interactions between particles of the same species. Starting from the definition of velocities obtained in \eqref{sigma_representation} and using the microscopic hypotheses, we derive the \textit{weak form} in the $\sigma$-representation, considering a test function $\psi: \R^d \to \R$.
We assume as in \cite{rey2012} that the collision kernel has the form
\begin{equation}
    B(|u|, \cos \theta, \mathcal{E}(f))=\Phi(|u|) \, b(u, \cos\theta, \mathcal{E}(f)),
\end{equation}
where $u=v-v_*$ is the relative velocity, $\theta$ is the angle between $\sigma$ and $u$ and $\mathcal{E}$ is the total kinetic energy, defined as
\begin{equation}
    \mathcal{E}(f) = \sum_{i=1}^M \int_{\R^d} m_i f_i(v) |v|^2 \, dv. 
\end{equation}
This choice automatically implies the microreversibility assumption if $b(u, \cos\theta, \mathcal{E}(f)) = b(-u, \cos\theta, \mathcal{E}(f))$.\\
Thus we obtain
\begin{equation}
	\begin{aligned}
	\label{weak_monospecies_1}
    \int_{\R^d} \mathcal{Q}_{i,i}^{inel} &(f,f) (v) \psi(v) \, dv= \\&=\int_{\R^d \times \R^d \times \mathbb{S}^{d-1}} f f_* (\psi'-\psi) B(|v-v_*|, \cos \theta, \mathcal{E}(f)) d\sigma \, dv \, dv_*.
    \end{aligned}
\end{equation}
If we snap $v$ and $v_*$ in \eqref{weak_monospecies_1} for a fixed $\sigma$, then $v'$ and $v'_*$ are swapped according to the $\sigma$-representation given in \eqref{sigma_representation}. Thus we obtain 
\begin{equation}
	\begin{aligned}
	\label{weak_monospecies_2}
    \int_{\R^d} \mathcal{Q}_{i,i}^{inel} &(f,f) (v) \psi(v) \, dv= \\ &=\int_{\R^d \times \R^d \times \mathbb{S}^{d-1}} f_* f (\psi'_*-\psi_*) B(|v-v_*|, \cos \theta, \mathcal{E}(f)) d\sigma \, dv \, dv_*,
    \end{aligned}
\end{equation}
which, combined with \eqref{weak_monospecies_1}, gives the following weak form
\begin{equation}
	\begin{aligned}
	\label{weak_monospecies}
	\int_{\R^d} \mathcal{Q}_{i,i}^{inel} &(f,f) (v) \psi(v) \, dv= \\ &=\frac{1}{2} \int_{\R^d \times \R^d \times \mathbb{S}^{d-1}} f_* f (\psi'+\psi'_*-\psi-\psi_*) B(|v-v_*|, \cos \theta, \mathcal{E}(f)) d\sigma \, dv \, dv_*.
	\end{aligned}
\end{equation}
If we consider the so-called \textit{generalized hard spheres} case, the collision kernel becomes
\begin{equation}
    B(|u|, \cos \theta, \mathcal{E}) = |u|^\lambda b(\cos \theta) \mathcal{E}^\gamma,
\end{equation}
where $\lambda \in [0,1]$ and the angular cross section $b$ verifies
\begin{equation}
    0 < \beta_1 \leq b(x) \leq \beta_2 < \infty, \qquad \forall \, x \in [-1,1].
\end{equation}
In particular, the choice $\lambda=0$ leads to the simplified \textit{Maxwellian pseudo-molecules} case and $\lambda=1$ to the more relevant \textit{hard spheres molecules} case. Recalling the definition of the pre-collisional velocities \eqref{sigma_representation}, this weak formulation of the mono-species operator leads to the conservation of the mass for each species. 
We can derive the weak form in the $\omega$-representation as well, following the same procedure and obtaining
\begin{equation}
	\begin{aligned}
  	\label{weak_monospecies_omega}
	\int_{\R^d} \mathcal{Q}_{i,i}^{inel} &(f,f) (v) \psi(v) \, dv= \\  &=\frac{1}{2} \int_{\R^d \times \R^d \times \mathbb{S}^{d-1}} f_* f (\psi'+\psi'_*-\psi-\psi_*) \tilde{B}(|v-v_*|, \cos \theta, \mathcal{E}(f)) d\omega \, dv \, dv_*,  
	\end{aligned}
\end{equation}
where now the post-collisional velocities are computed by \eqref{omega_representation}, see also \cite{carrillo2020recent}.

\subsubsection{Bi-species Collision Operator}
We would like to characterize the bi-species collision operators $\mathcal{Q}_{i, j}^{inel}$ for $i \neq j$. Therefore, considering the definition \eqref{sigma_representation}, in the same spirit of the mono-species framework, we obtain the following formulation of the collision operators in the $\sigma$-representation
\begin{equation}
	\label{weak_bispecies}
	\int_{\R^d} \mathcal{Q}_{i,j}^{inel} (f,g) (v) \psi(v) \, dv= \int_{\R^d \times \R^d \times \mathbb{S}^{d-1}} f g_* (\psi'-\psi) B(|v-v_*|, \cos \theta, \mathcal{E}(f)) d\sigma \, dv \, dv_*.
\end{equation}
In addition, if we consider both the bi-species collision operators, we can give the following characterization
\begin{equation}
	\label{weak_bispecies_sum}
	\begin{aligned}
		&\int_{\R^d} \mathcal{Q}_{i,j}^{inel} (f,g) (v) \psi(v) \, dv + \int_{\R^d} \mathcal{Q}_{j,i}^{inel} (g,f) (v) \phi(v) \, dv=\\&= \int_{\R^d \times \R^d \times \mathbb{S}^{d-1}} f g_* (\psi'+\phi'_*-\psi-\phi_*) B(|v-v_*|, \cos \theta, \mathcal{E}(f)) d\sigma \, dv \, dv_*,
	\end{aligned}
\end{equation}
for any $\psi, \phi: \R^d \to \R$. The collision kernel is assumed again of \textit{generalized hard sphere} type, with the same considerations of the previous section. \\
We recall that the weak formulations for mixtures involve the same collision kernel of the mono-species operator. We observe that the choice of $\psi(v)=m_i\, v$ and $\phi(v)=m_j \, v$ in \eqref{weak_bispecies_sum} leads to the conservation of the total momentum.
In this case too, it is possible to derive the weak form in the $\omega$-representation as
\begin{equation}
	\label{weak_bispecies_sum_omega}
	\begin{aligned}
		&\int_{\R^d} \mathcal{Q}_{i,j}^{inel} (f,g) (v) \psi(v) \, dv + \int_{\R^d} \mathcal{Q}_{j,i}^{inel} (g,f) (v) \phi(v) \, dv=\\&= \int_{\R^d \times \R^d \times \mathbb{S}^{d-1}} f g_* (\psi'+\phi'_*-\psi-\phi_*) \tilde{B}(|v-v_*|, \cos \theta, \mathcal{E}(f)) d\omega \, dv \, dv_*.
	\end{aligned}
\end{equation}

\subsection{Strong form}
As in \cite{carrillo2020recent}, we would like to derive a strong form of $\mathcal{Q}^{inel}$ in the $\omega$-representation. It is fundamental to understand that the derivation is completely different from the elastic case, since the transformation $(v,v_*,\sigma) \to (v',v'_*,\sigma)$ is not an involution, thus we have to go through the $\omega$-representation.\\
The derivation both for the mono-species collision operator and for the bi-species one is similar to the computation performed for the single-species granular gases equation and it is based on the expression of the pre-collisional velocities \eqref{pre_collisional} (see \cite{carrillo2020recent} for more details). Given a restitution coefficient $e=e(|u|)$ depending on the relative velocity of the particles, we first assume that
\begin{equation}
	J(|u|, \cos \theta) \neq 0, \qquad \forall \, \theta.
\end{equation}
Starting from \eqref{pre_collisional} in the $\omega$-representation, for a constant restitution coefficient $e \equiv \text{const}$, we obtain $J=e$.
Indeed, noticing that 
\begin{align*}
	|v'-v'_*|^2=|v-v_*|^2+\left(\frac{1}{e^2} -1 \right) ((v-v_*) \cdot \omega),
\end{align*}
we observe that the computation of the Jacobian is independent from the mass species and it is identical to the single-species case discussed in \cite{carlen2009, carrillo2020recent}.\\
In the case of \textit{generalized hard spheres}  collision kernel, it yields the following representation
\begin{equation}                   
    \mathcal{Q}^{inel}_{i,j} (f^i,f^j)(v) =\mathcal{Q}^{inel,+}_{i,j}(f^i,f^j)(v)-f^i(v)\, L(f^j)(v),
\end{equation}
where the loss part is given by
\begin{equation}
    L(f)(v)=\int_{\R^d \times \mathbb{S}^{d-1}} \tilde{B}(|v-v_*|, \cos \theta, \mathcal{E}(f)) f_* d\omega dv_*
\end{equation}
and the gain part is given by
\begin{equation}
    \mathcal{Q}^{inel,+}_{i,j} (f^i,f^j) (v) = \int_{\R^d \times \mathbb{S}^{d-1}} \tilde{\Phi}^+_e (|v-v_*|, \cos \theta) \tilde{b}^+_e(\cos \theta) \frac{\mathcal{E}^\gamma}{J(|v-v_*|,\cos \theta)} {'f^i} \,{'f^j_*} d\omega dv_*,
\end{equation}
where
\begin{equation}
    \tilde{\Phi}^+_e (r,s) = \Phi \left(\frac{r}{e} \sqrt{e^2 + (1-e^2) s^2} \right) = \left(\frac{r}{e} \sqrt{e^2 + (1-e^2) s^2} \right)^\lambda
\end{equation}
and
\begin{equation}
    \tilde{b}^+_e(s) = \tilde{b} \left( \frac{s}{\sqrt{e^2 + (1-e^2)s^2}} \right)
\end{equation}
Thus, as in the classical Boltzmann equation, the operator $\mathcal{Q}^{inel,+}_{i,j} (f^i,f^j) (v)$ is usually referred to as \textit{gain term}, since it expresses the number of particles of velocity $v$ created by collision of particles of pre-collisional velocities given by \eqref{pre_collisional}, whereas $f^i(v)L(f^j)(v)$ is usually known as \textit{loss} term, because it models the loss of particles of pre-collisional velocity equal to $v$.\\
It is possible to establish the strong form also in the $\sigma$-representation by changing variable in the operator, to find the expressions for the loss and gain terms and we refer to \cite{carlen2009} for more details. In the following sections, we will always use the weak formulation of the equation.

\subsection{Conservation properties}
As pointed out in Section  \ref{Subsection_Weak_Form}, using the weak formulation of the collision operators, we are able to recover some conservation properties of the multi-species Boltzmann equation for granular gases. In particular, the choice $\psi(v)$ in \eqref{weak_monospecies} and \eqref{weak_bispecies} allows us to recover the conservation of the mass of the single species $i$. Furthermore, for the bi-species case, considering $\psi(v)=m_i\, v$ and $\phi(v)=m_j \, v$ in \eqref{weak_bispecies_sum}, we recover the conservation of the total momentum.\\
Let us define the following macroscopic quantities from the distribution function
\begin{equation}
\label{macro_quantities}
    n_i= \int d\uv f_i(\uv), \qquad \rho \uv= \sum_{i} \int d\uv \, m_i f_i(\uv),
\end{equation}
the number densities $n_i$, the total density $\rho$ and the flow velocity $\uv$.

\subsubsection{Macroscopic properties of the granular gases mixtures operator}
The main difference between a mixture of granular gases and a mixture of perfect gases is the dissipation of the total kinetic energy. Indeed, using the weak formulation of the inelastic collision operators \eqref{weak_monospecies}-\eqref{weak_bispecies}, we obtain
\begin{equation}
	\sum_{i,j} \int_{\R^d} \mathcal{Q}_{i,j}^{inel} (v) \begin{pmatrix}
		1 \\ m_i \, v \\ \displaystyle m_i |v|^2/2
	\end{pmatrix} dv = 
	\begin{pmatrix}
		0 \\ 0 \\ -\displaystyle\sum_{i,j} D_{ij}(f_i,f_j) 
	\end{pmatrix}
\end{equation}
where $D_{ij}(f_i,f_j) \geq 0$ are the energy dissipation functionals which depend only on the collision kernel. They read as
\begin{equation}
    D_{ij}(f_i,f_j) := \int_{\R^d \times \R^d} f_i\,f_{j*} \Delta_{ij} (|v-v_*|, \mathcal{E}(f)) dv\, dv_*,
\end{equation}
where $\Delta_{ij}$ are the \textit{energy dissipation rates}, computed using the weak formulation and the microscopic dynamics \eqref{microscopic_dynamics}, given by
\begin{equation}
    \Delta_{ij}(|u|, \mathcal{E}) := \frac{m_i\,m_j}{(m_i+m_j)} \frac{1-e^2}{2} \int_{\mathbb{S}^{d-1}} |u \cdot \omega|^2 B(|u|, \cos \theta, \mathcal{E}) \, d\omega \geq 0, \qquad \forall \, e \in [0,1].
\end{equation}
In the following section we show that it is not possible to obtain any entropy dissipation for this equation. However, it is possible to exploit the particular structure of the entropy functional to recover the Haff's cooling law \cite{haff1983}, as we will show in Section \ref{Section_Haffs_Law}.\\
Formally, the dissipation of kinetic energy implies that as $\eps \to 0$, then the particle distribution functions $f^i$ converge towards Dirac distributions centered on the mean velocity $u_i$, which is the same for each species ($u_i \equiv u_j$):
\begin{equation}
\label{dirac_convergence}
    f^i(t,x,v)=\rho_i(x) \,\delta_0 (v-u_i(x)), \qquad \forall \, (x,v) \in \R^d \times \R^d.
\end{equation}
In the multi-velocity case it is possible to have a limit distribution consisting of a sum of Dirac deltas centered on a different mean velocity for each species.

\section{The Homogeneous Equation}
\label{Section_Homogeneous}
Let us consider the following homogeneous multi-species inelastic Boltzmann equation,
\begin{equation}
    \label{homogeneous_BEs}
	\begin{cases}
		\partial_t f_i = \displaystyle \mathcal{Q}^{\, inel \,}_i(\f),\\
		\mathcal{Q}^{inel}_i (\f) = \displaystyle \sum_{j=1}^N \mathcal{Q}_{i,j} (f_i,f_j),\\ 
		f_i(0,\x,\uv)= f_{i}^{\,0} (\x,\uv).
	\end{cases}
\end{equation}
The Boltzmann-type equation is complemented with the initial condition $f^{\,0}$, which is supposed to satisfy the following moment conditions
\begin{equation}
	\label{hypothesis_initial_moments}
	\int_{\R^d} m_i f^{\,0}_i(v)\, dv = 1, \qquad \sum_{i=1}^M \int_{\R^d} m_i f^{\,0}_i (v) v \, dv = 0.
\end{equation}
Note that the assumption \eqref{hypothesis_initial_moments} is not restrictive, thanks to the conservation properties.
\subsection{Functional Framework}
Let us present some functional spaces needed in the paper. We denote, for any $q \in \mathbb{N}$, the Banach space
\begin{equation}
    L^1_q = \left \{  f: \R^d  \to \R^M \,\, \text{ measurable } \, \text{ s.t. } \, ||f||_{L^1_q} := \displaystyle \sum_{i=1}^M \int_{\R^d} |f_i(v)|(1+m_i^{q/2}|v|^q)\,dv < \infty \right \}.
\end{equation}
Note that we have incorporated the species mass inside the usual chinese bracket.\\
We also denote by $W_q^{1,1}$ the weighted Sobolev space defined as
\begin{equation}
    W_q^{1,1} := \left \{ f \in L^1_q, \, \nabla f \in L^1_q \right \}.
\end{equation}
We can now introduce the space $BV_q(\R^d)$ of weighted bounded variation functions for $q \in \mathbb{N}$ as the set of the weak limits in $\mathcal{D}'(\R^d)$ of sequences of smooth functions bounded in $W_q^{1,1}$. Finally, we denote by $\mathcal{M}^1(\R^d)$ the space of normalized probability measures on $\R^d$.\\
For the sake of simplicity, we define 
\begin{equation}
    \langle v \rangle_i = (1+m_i|v|^2)^{1/2},
\end{equation}
and we notice that the weighted norm $|| \cdot ||_{L^1_q}$ is equivalent to the weighted norm  induced by $\langle \cdot \rangle^q$.

\subsection{Entropy for Multi-species Granular Gases}
The dissipation of energy implies an anomalous behavior for the solution of multi-species granular gases equation, because it implies an explosive behavior for the solution.\\
As we pointed out, in the case of mixtures the distribution function $f^i$ relative to species $i$ converges formally towards a Dirac delta, centered in the mean velocity $u_i$.
As for the single-species case, the granular framework makes impossible to recover any classical results for entropy. Indeed, the lack of conservation and the consequent loss of involutive properties in the transformation $(v,v_*,\sigma) \to (v',v'_*,\sigma)$ imply that it is not possible to obtain any entropy dissipation for this equation. For the sake of simplicity, we restrict to the case of two-species mixtures.\\
We first define the entropy functional for gas mixtures, consistent with the classical Boltzmann $\mathcal{H}$-functional
\begin{equation}
    \mathcal{H}(f) = \int_{\R^d} f_1(v)\log f_1(v)\,dv + \int_{\R^d} f_2(v)\log f_2(v)\,dv.
\end{equation}
In order to avoid unessential multiplicative constants we assume $\eps=1$ without loss of generality. The time derivative of the entropy functional in space homogeneous conditions reads as
\begingroup
\allowdisplaybreaks
\begin{align*}
		\frac{d}{dt} \mathcal{H}(f)&= \int_{\R^d} \mathcal{Q}_{1,1}^{inel} (f_1,f_1) \log f_1 \, dv + \int_{\R^d} \mathcal{Q}_{1,2}^{inel} (f_1,f_2) \log f_1 \, dv \\
		&+\hphantom{\frac{1}{2}} \int_{\R^d} \mathcal{Q}_{2,1}^{inel} (f_2,f_1) \log f_2 \, dv + \int_{\R^d} \mathcal{Q}_{2,2}^{inel} (f_2,f_2) \log f_2 \, dv\\
		&= \frac{1}{2} \int_{\R^d \times \R^d \times \mathbb{S}^{d-1}} f_{1*} f_1 \log \left(\frac{f_1' f_{1*}'}{f_1 f_{1*}} \right) B d\sigma \, dv \, dv_* \\
        &+\hphantom{\frac{1}{2}} \int_{\R^d \times \R^d \times \mathbb{S}^{d-1}} f_1 f_{2*} \log \left(\frac{f_1' f_{2*}'}{f_1 f_{2*}} \right) B d\sigma \, dv \, dv_*\\
		&+\frac{1}{2} \int_{\R^d \times \R^d \times \mathbb{S}^{d-1}} f_{2*} f_2 \log \left(\frac{f_2' f_{2*}'}{f_2 f_{2*}} \right) B d\sigma \, dv \, dv_* \\
		&=\frac{1}{2} \int_{\R^d \times \R^d \times \mathbb{S}^{d-1}} f_{1*} f_1 \left[\log \left(\frac{f_1' f_{1*}'}{f_1 f_{1*}} \right) - \frac{f_1' f_{1*}'}{f_1 f_{1*}} + 1 \right] B d\sigma \, dv \, dv_* \\
        &+\frac{1}{2} \int_{\R^d \times \R^d \times \mathbb{S}^{d-1}} (f_{1*}' f_1' - f_{1*} f_1) B d\sigma \, dv \, dv_* \\
		&+ \phantom{\frac{1}{2}}\int_{\R^d \times \R^d \times \mathbb{S}^{d-1}}f_1 f_{2*}  \left[\log \left(\frac{f_1' f_{2*}'}{f_1 f_{2*}} \right) - \frac{f_1' f_{2*}'}{f_1 f_{2*}} + 1 \right] B d\sigma \, dv \, dv_* \\
        &+\phantom{\frac{1}{2}}\int_{\R^d \times \R^d \times \mathbb{S}^{d-1}} (f_{2*}' f_1' - f_{2*} f_1) B d\sigma \, dv \, dv_* \\
		&+\frac{1}{2} \int_{\R^d \times \R^d \times \mathbb{S}^{d-1}} f_{2*} f_2 \left[\log  \left(\frac{f_2' f_{2*}'}{f_2 f_{2*}} \right) - \frac{f_2' f_{2*}'}{f_2 f_{2*}} + 1 \right] B d\sigma \, dv \, dv_* \\
        &+\frac{1}{2} \int_{\R^d \times \R^d \times \mathbb{S}^{d-1}} (f_{2*}' f_2' - f_{2*} f_2) B d\sigma \, dv \, dv_*.
\end{align*}
\endgroup
The first terms in each row of the equality are non-positive, since $\log x - x +1 \leq 0$. This is the same property used in the one-species elastic Boltzmann $H$-theorem, since it represents the elastic contribution of the collision operator.\\
Finally we have
\begin{equation}
	\label{entropy_functional}
	\begin{aligned}
		\mathcal{H}(f) \leq & \phantom{+} \frac{1}{2} \int_{\R^d \times \R^d \times \mathbb{S}^{d-1}} (f_{1*}' f_1' - f_{1*} f_1) B d\sigma \, dv \, dv_* \\
		&+\int_{\R^d \times \R^d \times \mathbb{S}^{d-1}} (f'_{2*} f'_1 - f_{2*} f_1)B d\sigma \, dv \, dv_* \\
        & +\frac{1}{2} \int_{\R^d \times \R^d \times \mathbb{S}^{d-1}} (f_{2*}' f_2' - f_{2*} f_2) B d\sigma \, dv \, dv_*.
	\end{aligned}
\end{equation}
All these terms have no sign \textit{a priori} and they are $0$ only in the elastic case, due to the involutive nature of the collisional transformation $(v,v_*) \to (v',v'_*)$.\\
Therefore, as in the case of single-species granular gases, Boltzmann's entropy $\mathcal{H}(f)$ is not dissipated by the solution if the restitution coefficient $e < 1$.\\

\subsection{Kinetic Energy Dissipation}
We can nevertheless define some macroscopic quantities in order to study the large time behavior of the solution, since we have no information about the entropy. If we assume that the collision kernel is of the general type, using polar coordinates we can rewrite the dissipation rate as
\begin{equation}
    \Delta_{ij}(|u|, \mathcal{E}) = \tau \frac{m_i\,m_j}{(m_i+m_j)} \frac{1-e^2}{2} |u|^{\lambda+2} \mathcal{E}^\gamma,
\end{equation}
where
\begin{equation}
    \tau = |\mathbb{S}^{d-2}| \int_0^\pi \cos^2(\theta)\, \sin^{d-3} (\theta) b (\cos(\theta)) d\theta < \infty.
\end{equation}
Let us assume the non-negativity of the distribution functions $f_i$ and $f_j$. Then, we have
\begin{equation*}
\begin{aligned}
    \frac{d}{dt} \mathcal{E}(f)(t) = & -\tau m_i \frac{1-e^2}{4} \mathcal{E}(f)^\gamma(t) \int_{\R^d \times \R^d} f_{i*} f_i |v-v_*|^{\lambda+2} dv\,dv_* \\ & -\tau m_j \frac{1-e^2}{4} \mathcal{E}(f)^\gamma(t) \int_{\R^d \times \R^d} f_{j*} f_j |v-v_*|^{\lambda+2} dv\,dv_* \\  & -\tau \frac{m_i\,m_j}{(m_i+m_j)} \frac{1-e^2}{2} \mathcal{E}(f)^\gamma(t) \int_{\R^d \times \R^d} f_i f_{j*} |v-v_*|^{\lambda+2} dv\,dv_*  \\
    \leq & -\tau \frac{m_i}{m_i+m_j} \frac{1-e^2}{4} \mathcal{E}(f)^\gamma (t) \int_{\R^d \times \R^d} f_i [m_i f_{i*} + m_j f_{j*}] |v-v_*|^{\lambda+2} dv\,dv_* \\
    & -\tau \frac{m_j}{(m_i+m_j)} \frac{1-e^2}{4} \mathcal{E}(f)^\gamma (t) \int_{\R^d \times \R^d} f_{j*} [m_i f_i + m_j f_j] |v-v_*|^{\lambda+2} dv\,dv_*  \\
    \leq & - \tau \frac{1}{(m_i+m_j)} \frac{1-e^2}{4} \mathcal{E}(f)^\gamma(t) \int_{\R^d} [m_i\,f_i + m_j \, f_j] |v|^{\lambda+2} dv  \\
    \leq & - \tau \frac{1}{(m_i+m_j)} \frac{1-e^2}{4} \mathcal{E}(f)^{\gamma+1+\lambda/2}(t),
\end{aligned}   
\end{equation*}
where we have used Jensen's inequality to estimate
\begin{equation*}
    \int_{\R^d} [m_i\,f_{i*} + m_j\,f_{j*}] |v-v_*|^{\lambda+2} dv_* \geq \Big|v-\int_{\R^d} [m_i\,f_{i*} v_* + m_j \, f_{j*} v_*] dv_* \Big|^{\lambda+2} = |v|^{\lambda+2}.
\end{equation*}
Therefore, setting 
\begin{equation}
    C_e = \tau \rho \, \frac{1-e^2}{4\,(m_i+m_j)} \quad \text{ and } \quad \alpha=\gamma+\frac{1}{2},
\end{equation}
we obtain the following large time behaviors:
\begin{itemize}
    \item pseudo-Maxwellian molecules ($\lambda=\gamma=0$) decays exponentially fast towards Dirac delta:
    \begin{equation}
    	\mathcal{E}(f)(t) \leq \displaystyle \mathcal{E}(f^{in}) e^{-C_e\,t}
    \end{equation}
    \item Hard spheres ($\lambda=1$,  $\gamma=0$) exhibits the Haff's cooling Law (see \cite{haff1983}):
    \begin{equation}
    	\label{Haffs_Law_Hardspheres}
        \mathcal{E}(f)(t) \leq \displaystyle \left(\mathcal{E}(f^{in})^{-\frac{1}{2}} + C_e \frac{t}{2} \right)^{-2}
    \end{equation}
\end{itemize}
We have proven the following formal result
\begin{prop}
    \label{proposition_energy_dissipation}
    Let $(f_i^0)$ be a set of nonnegative distribution functions satisfying \eqref{hypothesis_initial_moments} and let $f_i(t,v)$ be the associated solution to the Cauchy problem \eqref{homogeneous_BEs}, with a constant restitution coefficient. Then in the VHS case,
    \begin{equation}
    \label{energy_dissipation}
        \frac{d}{dt} \mathcal{E}(f)(t) \leq - \tau \frac{1}{m_i+m_j} \frac{1-e^2}{4} \left(\mathcal{E}(f)(t)\right)^{\gamma+1+\lambda/2} \qquad \forall \, t \geq 0.
    \end{equation}
\end{prop}
\begin{rem}
It is possible to extend this result, giving an estimation on the decay of $\mathcal{E}(f)(t)$ also in the case of non constant restitution coefficients in a weakly inelastic regime, we refer to \cite{alonso2010} for the proof in the single-species case.
\end{rem}

\section{A Povzner-type Inequality for Granular Mixtures}
\label{Section_Povzner_inequality}
The key ingredient for obtaining estimates in the space $L^1_q$ is the so-called Povzner-type inequality \cite{povzner1962}, that we will extend to the multi-species inelastic case.\\
These estimates are well known for the Boltzmann equation with elastic collisions, see \cite{bobylev1997, mischler1999, lu1999}, as well for the inelastic case \cite{bobylev2004, gamba2004}. They have been extended also to the multi-species case with elastic collisions in a more recent work \cite{briant2016boltzmann}. Our proof is an extension of \cite{gamba2004}, in which the authors prove a rougher version of the Povzner-type inequality without exploiting the symmetric properties of the collisions. Since the multi-species case is characterized by an intrinsic asymmetry due to the presence of difference masses, this will be the best approach to extend such results to our case. \\
We consider $\psi(x)$, $x>0$ to be a convex non-decreasing function, satisfying the following assumptions
\begin{gather}
	\psi(x) \geq 0, \quad x >0, \quad \psi(0)=0; \label{assumptions_psi_1}\\[6pt]
	\psi(x) \text{ is convex }, \, C^1([0, \infty)), \,\, \psi''(x) \text{ is locally bounded};\\[6pt]
	\psi'(\alpha x) \leq \eta_1(\alpha) \psi'(x), \quad x>0, \quad \alpha >1;\\[6pt]
	\psi''(\alpha x) \leq \eta_2(\alpha) \psi''(x), \quad x>0, \quad \alpha >1, \label{assumptions_psi_last}
\end{gather}
where $\eta_1(\alpha)$ and $\eta_2(\alpha)$ are bounded functions of $\alpha$ for every $\alpha >0$.\\
We are interested in estimating the expression
\begin{equation}
    q[\psi](v,v_*,\sigma) = \psi(m_i \, |v'|^2)+\psi(m_j \, |v'_*|^2)-\psi(m_i \, |v|^2)-\psi(m_j \, |v_*|^2)
\end{equation}
and the expression
\begin{equation}
\label{povzner_integral_expression}
    \bar{q}[\psi](v,v_*) = \int_{\mathbb{S}^{d-1}} q[\psi](v,v_*,\sigma) b(v-v_*, \sigma) d\sigma.
\end{equation}
The aim is to treat the case of
\begin{equation}
    \psi(x)=x^p, \qquad \psi(x)=(1+x)^p-1, \quad p>1,
\end{equation}
and possibly truncated versions of such functions. Let us recall the following Lemma which will be useful to establish our result.
\begin{lem}[\cite{gamba2004}, Lemma $3.1$]
    \label{lemma_properties_psi}
    Assume that $\psi(x)$ satisfies \eqref{assumptions_psi_1}-\eqref{assumptions_psi_last}. Then
    \begin{equation}
        \psi(x+y) - \psi(x) - \psi(y) \leq A(x \psi'(y)+y\psi'(x))
    \end{equation}
    and
    \begin{equation}
        \psi(x+y) - \psi(x) - \psi(y) \geq b\, xy \psi''(x+y),\\
    \end{equation}
    where $A=\eta_1(2)$ and $b=(2\eta_2(2))^{-1}$.
\end{lem}
In order to obtain some bounds, let us rewrite the post-collisional velocities $v'$ and $v'_*$ in a more convenient way, with respect to the center of mass-relative velocity variables. We set
\begin{equation}
    v'=\frac{w+m_j \lambda |u| \omega}{m_i+m_j}, \qquad v'_*=\frac{w-m_i \lambda |u| \omega}{m_i+m_j},
\end{equation}
where $w=m_i\,v+m_j\,v_*$, $u=v-\,v_*$ and $\omega$ is a parameter vector on the sphere $\mathbb{S}^{d-1}$. The parameter $\lambda$ is chosen s.t.
\begin{equation}
    \lambda \omega = \beta \sigma + (1-\beta) \nu,
\end{equation}
where $\beta=\displaystyle \frac{1+e}{2}$ and $\nu=\displaystyle \frac{u}{|u|}$. Therefore, using the relation
\begin{equation}
    \beta \sigma = \lambda \omega - (1-\beta) \nu,
\end{equation}
and taking the dot product of each side with itself, we obtain
\begin{equation}
    \lambda^2 - 2 \lambda (1-\beta) \cos \chi + (1-\beta)^2 -\beta^2 = 0,
\end{equation}
where $\chi$ is the angle between $u$ and $\omega$. It leads to
\begin{equation}
    \lambda= (1-\beta) \cos \chi + \sqrt{(1-\beta)^2(\cos^2\chi -1) + \beta^2}.
\end{equation}
Notice that
\begin{equation}
    0 < e \leq \lambda(\cos\chi) \leq 1,
\end{equation}
for all $\chi$. Then with this parametrization we have
\begin{equation}
    \begin{aligned}
        &|v'|^2 = \frac{|w|^2 + m_j^2 \lambda^2 |u|^2 + 2 \lambda m_j |u| |w| \cos \mu}{(m_i+m_j)^2},\\
        &|v'_*|^2 = \frac{|w|^2 + m_i^2 \lambda^2 |u|^2 - 2 \lambda m_i |u| |w| \cos \mu}{(m_i+m_j)^2},
    \end{aligned}
\end{equation}
where $\mu$ is the angle between $w$ and $\omega$.
\begin{lem}
\label{lemma_povnzer_2}
    Assume that the function $\psi$ satisfies \eqref{assumptions_psi_1}-\eqref{assumptions_psi_last}. Then we have
    \begin{equation}
        q[\psi]=-n[\psi]+p[\psi],
    \end{equation}
    where 
    \begin{equation}
        p[\psi] \leq A(m_i |v|^2 \psi'(m_j\,|v_*|^2) + m_j |v_*|^2 \psi'(m_i\,|v|^2))
    \end{equation}
    and
    \begin{equation}
        n[\psi] \geq \kappa(\lambda, \mu) (m_i |v|^2 + m_j |v_*|^2)^2 \psi''(m_i |v|^2 + m_j |v_*|^2).
    \end{equation}
    Here $A=\eta_1(2)$, $b=2\left(\eta_2(2)\right)^{-1}$ and 
    \begin{equation}
        \kappa(\lambda, \mu) = b \frac{m_i\,m_j}{(m_i+m_j)^2} \lambda^4 (\eta_2(\lambda^{-2}))^{-1} \sin^2 \mu.
    \end{equation}.
\end{lem}
\begin{proof}
We start by setting
\begin{equation}
    p[\psi] = \psi(m_i |v|^2 + m_j |v_*|^2)-\psi(m_i |v|^2) - \psi(m_j |v_*|^2)
\end{equation}
and
\begin{equation}
    n[\psi] = \psi(m_i |v|^2 + m_j |v_*|^2)-\psi(m_i |v'|^2) - \psi(m_j |v'_*|^2)
\end{equation}
From Lemma \ref{lemma_properties_psi}, the estimates for $p$ easily follows. It remains to verify the lower bound for $n$, where we use that $m_i |v|^2+ m_j |v_*|^2 \geq m_i |v'|^2+ m_j|v'_*|^2$. We then obtain
\begin{equation}
    \begin{aligned}
        n[\psi] &\geq \psi(m_i |v'|^2 + m_j |v'_*|^2)-\psi(m_i |v'|^2) - \psi(m_j |v'_*|^2) \\
        &\geq b \, m_i m_j |v'|^2 |v'_*|^2 \psi''(m_i |v'|^2 + m_j |v'_*|^2) \\
        &= b \, \xi(v',v'_*) (m_i |v'|^2 + m_j |v'_*|^2)^2 \psi''(m_i |v'|^2 + m_j |v'_*|^2),
    \end{aligned}
\end{equation}
where 
\begin{equation}
    \xi(v',v'_*)=\frac{m_i |v'|^2}{m_i|v'|^2+m_j|v'_*|^2} \frac{m_j |v'_*|^2}{m_i|v'|^2+m_j|v'_*|^2}.
\end{equation}
Now we compute
\begin{equation*}
    \begin{aligned}
        (m_i|v'|^2+m_j|v'_*|^2)^2\xi(v',v'_*) &= m_i\,m_j |v'|^2 |v'_*|^2 \\
        &=\frac{m_i\,m_j}{(m_i+m_j)^4} \Big( |w|^4 + m_i^2 \lambda^2 |u|^2 |w|^2 - 2 \lambda m_i |u| |w|^3 \cos \mu \\
        &+ m_j^2 \lambda^2 |u|^2 |w|^2 + m_i^2 m_j^2 \lambda^4 |u|^4 - 2 m_i m_j^2 \lambda^3 |u|^3 |w| \cos \mu \\
        &+2 \lambda m_j |u| |w|^3 \cos \mu + 2 m_i^2 m_j \lambda^3 |u|^3 |w| \cos \mu - 4 m_i m_j \lambda^2 |u|^2 |w|^2 \cos^2 \mu \Big)\\
        &=\frac{m_i\,m_j}{(m_i+m_j)^4} \Big( \left(m_i |v'|^2 + m_j |v'_*|^2\right)^2 (m_i+m_j)^2 \\
        &+\lambda^2|u|^2|w|^2(m_i-m_j)^2 + 2\lambda |u||w| \cos \mu (m_i-m_j) (m_i\,m_j \lambda^2 |u|^2 - |w|^2) \\
        &-4m_i\,m_j \lambda^2 |u|^2 |w|^2 \cos^2 \mu \Big) \\
        &=\frac{m_i\,m_j}{(m_i+m_j)^4} \Big( (m_i |v'|^2 + m_j |v'_*|^2)^2 (m_i+m_j)^2 \\
        &+(\lambda|u||w|(m_i-m_j)-\cos \mu (m_i\,m_j \lambda^2 |u|^2 - |w|^2))^2 \\
        &+ \cos^2 \mu (m_i\,m_j \lambda^2 |u|^2 - |w|^2)^2  - 4m_i\,m_j \lambda^2 |u|^2 |w|^2 \cos^2 \mu \Big).
    \end{aligned}
\end{equation*}
Since the second and the third terms of the sum inside the parentheses are positive, we obtain
\begin{equation}
    \begin{aligned}
    \xi(v',v'_*) &\geq \frac{m_i\,m_j}{(m_i+m_j)^2} \Big(1- \frac{4\,m_i\,m_j \lambda^2 |u|^2|w^2}{(m_i\,m_j\lambda^2|u|^2+|w|^2)^2}\cos^2\mu \Big) \\
    &\geq\frac{m_i\,m_j}{(m_i+m_j)^2}(1-\cos^2\mu)=\frac{m_i\,m_j}{(m_i+m_j)^2} \sin^2 \mu.
    \end{aligned}
\end{equation}
Finally, we notice that
\begin{equation*}   
	\begin{aligned} m_i|v'|^2+m_j|v'_*|^2=&\frac{m_i\,m_j\lambda^2|u|^2+|w|^2}{(m_i+m_j)} \\ &\geq\frac{1}{(m_i+m_j)} \lambda^2 \left(m_i\,m_j|u|^2+|w|^2\right) = \lambda^2 \left(m_i|v|^2+m_j|v_*|^2\right),
	\end{aligned}
\end{equation*}
since
\begin{align*}
    \frac{1}{(m_i+m_j)} &(m_i\,m_j|u|^2+|w|^2) \\ 
    =&\frac{1}{(m_i+m_j)} \left[m_i\,m_j (|v|^2 - \innerproduct{v}{v_*} + |v_*|^2) + m_i^2 |v|^2 + m_j^2 |v_*|^2 + m_i\,m_j \innerproduct{v}{v_*} \right] \\
    =&\frac{1}{(m_i+m_j)} \left[m_i(m_j+m_i)|v|^2 + m_j(m_i+m_j)|v_*|^2\right].
\end{align*}
Thus, we conclude
\begin{equation*}
    n[\psi] \geq b \frac{m_i\,m_j}{(m_i+m_j)^2} 
    \lambda^4 \left(\eta_2(\lambda^{-2})\right)^{-1} \sin^2 \mu \left(m_i|v|^2+m_j|v_*|^2 \right)^2 \psi''\left(m_i|v|^2+m_j|v_*|^2 \right).
\end{equation*}
This completes the proof of the Lemma.
\end{proof}
This gives the formulation of the \textit{Povzner-type inequality} for the class of convex test functions $\psi$ satisfying \eqref{assumptions_psi_1}-\eqref{assumptions_psi_last} in the case of granular mixtures. 
\begin{lem}
\label{lem_cos}
    Assume that the function $\psi$ satisfies \eqref{assumptions_psi_1}-\eqref{assumptions_psi_last}. Then we have 
    \begin{multline*}
        \bar{q}[\psi] \leq - k \left(m_i |v|^2 + m_j |v_*|^2 \right)^2 \psi''\left(m_i |v|^2 + m_j |v_*|^2 \right) \\+ A \left(m_i |v|^2 \psi' (m_j |v_*|^2) + m_j |v_*|^2 \psi' (m_i |v|^2) \right),
    \end{multline*}
    where the constant $A$ is in Lemma \ref{lemma_povnzer_2} and $k>0$ is a constant depending on $\psi$.
\end{lem}
The proof of this Lemma is a straightforward extension of \cite[Lemma 3.3]{gamba2004} to the multi-species case, using the estimates obtained in Lemma \ref{lemma_povnzer_2}.\\
\begin{lem}
    Let be $p>1$, taking $\psi(x)=x^p$ yields
    \begin{equation*}
        |u|\bar{q}[\psi](v,v_*) \leq -k_p(m_i |v|^{2p+1} + m_j |v_*|^{2p+1}) + A_p (m_i\,m_j^{p}|v| |v_*|^{2p} + m_j\,m_i^p |v|^{2p} |v_*|).
    \end{equation*}
    Moreover, taking $\psi(x)=(1+x)^p-1$ yields
    \begin{equation*}
        |u|\bar{q}[\psi](v,v_*) \leq -k_p(m_i \langle v \rangle^{2p+1} + m_j \langle v_* \rangle^{2p+1}) + A_p (m_i\,m_j^{p} \langle v \rangle \langle v_* \rangle^{2p} + m_j\,m_i^p \langle v \rangle^{2p} \langle v_* \rangle).
    \end{equation*}
\end{lem}
\begin{proof}
    We use Lemma \ref{lem_cos} and the inequalities
    \begin{equation*}
        \Big||v|-|v_*|\Big| \leq |u| \leq |v|+|v_*|.  
    \end{equation*}
    In the case $\psi(x)=x^p$ the bounds have the form
    \begin{multline*}
        -p(p-1)k_p |u| \left(m_i |v|^2 + m_j |v_*|^2 \right)^p + \\ p A_p |u| \left(m_i |v|^2 m_j^{p-1} |v_*|^{2p-2} + m_j |v_*|^2 m_i^{p-1} |v|^{2p-2} \right),
    \end{multline*}
    where the constant $A_p$ depends on the function $\psi$ chosen.
    The terms appearing with the negative sign are estimated using the inequality
    \begin{equation*}
    	\begin{aligned}
        |u| \left(m_i |v|^2 + m_j |v_*|^2 \right)^p &\geq m_i^p |u| |v|^{2p} + m_j^p |u| |v_*|^{2p} \\
         &\geq m_i^p (|v|^{2p+1} - |v|^{2p}|v_*|) + m_j^p (|v_*|^{2p+1} - |v||v_*|^{2p}).
        \end{aligned}
    \end{equation*}
    For the other terms we have
    \begin{equation*}
        |u| \left(m_i |v|^2 m_j^{p-1} |v_*|^{2p-2} + m_j |v_*|^2 m_i^{p-1} |v|^{2p-2} \right) \leq C_p (|v||v_*|^{2p} + |v|^{2p}|v_*|).
    \end{equation*}
    Indeed, it holds 
    \begin{equation*}
        \begin{aligned}
            &|u| \left(m_i |v|^2 m_j^{p-1} |v_*|^{2p-2} + m_j |v_*|^2 m_i^{p-1} |v|^{2p-2} \right) \\
            & \leq \Big(m_i |v|^3 m_j^{p-1} |v_*|^{2p-2} + m_i |v|^2 m_j^{p-1} |v_*|^{2p-1}  \\ 
            &+m_j |v_*|^2 m_i^{p-1} |v|^{2p-1}+m_j |v_*|^3 m_i^{p-1} |v|^{2p-2} \Big) \\
            & \leq m_i m_j \left(m_j^{\frac{p^2-2p}{p-1}}|v||v_*|^{2p} + m_i^{\frac{p^2-2p}{p-1}} |v_*| |v|^{2p} + m_j^{\frac{2p^2-4p}{2p-1}} |v||v_*|^{2p} + m_i^{\frac{2p^2-4p}{2p-1}} |v_*| |v|^{2p}\right).
        \end{aligned}  
    \end{equation*} 
    The case $\psi(x)=(1+x)^p-1$ can be faced with the same strategy, by using the inequalities
    \begin{align*}
        (m_i|v|^2+m_j|v_*|^2)^2 \geq \frac{1}{2} (1+m_i|v|^2+m_j|v_*|^2)^2-1, \\ m_i|v| \geq (1+m_i|v|^2)^{1/2}-1.
    \end{align*}
\end{proof}
The Povzner-type inequalities of the previous Lemma could help studying the propagation and appearance of moments in \eqref{homogeneous_BEs}. We introduce the notation
\begin{equation}
    Y_i^s(t) = \int_{\R^d} f_i \langle v \rangle_i ^s \, dv,
\end{equation}
and the total quantity
\begin{equation}
    Y^s(t)=\sum_{i=1}^M Y_i^s(t) = \sum_{i=1}^M \int_{\R^d} f_i \langle v \rangle_i ^s \, dv.
\end{equation}
Let us notice that 
\begin{equation*}
    \begin{aligned}
        &\langle v' \rangle_i^s +  \langle v'_* \rangle_j^s - \langle v \rangle_i^s - \langle v_* \rangle_j^s \\
         &=(1+m_i|v'|^2)^{\frac{1}{2}} + (1+m_j|v'_*|^2)^{\frac{1}{2}} - (1+m_i|v|^2)^{\frac{1}{2}} - (1+m_j|v_*|^2)^{\frac{1}{2}} \\
        &=q[\psi](v,v_*,\sigma),
    \end{aligned}
\end{equation*}
for $\psi(x)=(1+x)^p-1$, with $p=\frac{s}{2}$.
\begin{lem}
\label{lemma_moments_estimates}
    Let $f$ be a sufficiently regular and rapidly decaying solution to the Boltzmann equation \eqref{homogeneous_BEs}. Then, the following inequality holds:
    \begin{equation*}
        \frac{d}{dt} Y^s \leq - \sum_{i=1}^M \left(2k_s m_i (Y_i^s)^{(s+1)/s} - K_1^s Y_i^s \right) - \sum_{i=1}^M \sum_{\substack{j=1\\j \neq i}}^M \left(k_s (m_i+m_j) (Y_i^s)^{(s+1)/s} - K_2^s Y_i^s \right),
    \end{equation*}
    where $k_s$,$K_1^s$ and $K_2^s$ are non-negative constants. In addition, we have
    \begin{equation*}
        \sup_{t>0} Y^s(t) \leq Y_*^s = \max \left\{Y^s(0), (K_{ii}^s/(2k_s\,m_i))^s, (K_{ij}^s/(k_s\,(m_i+m_j))^s \right\}.
    \end{equation*}
\end{lem}

\begin{proof}
    Using the weak formulation, with $\psi(v) = \langle v \rangle_i ^s$ we obtain
    \begin{multline}
        \frac{d}{dt} \sum_{i=1}^M \int_{\R^d} f_i \, \langle v \rangle_i^s dv \\ = \sum_{i=1}^M \int_{\R^d} \mathcal{Q}_{i,i}(f_i,f_i) \, \langle v \rangle_i^s \, dv + \sum_{i, j=1}^N\int_{\R^d} \mathcal{Q}_{i,j}(f_i,f_j) \, \langle v \rangle_i^s \, dv.
    \end{multline}
    Estimating the moments of the intra-species collision integral term 
    we obtain
    \begin{equation}
        \begin{aligned}
            \int_{\R^d} Q_{i,i}(f_i,f_i) \, \langle v \rangle_i ^s \, dv \leq - 2 k_s m_i Y_i^{s+1} + 2 A_s m_i Y_i^1 Y_i^s.
        \end{aligned}
    \end{equation}
    On the other hand, for the inter-species collision integral terms we have 
    \begin{multline}
            \int_{\R^d} Q_{i,j}(f_i,f_j) \, \langle v \rangle_i ^s \, dv + \int_{\R^d} Q_{j,i}(f_j,f_i) \, \langle v \rangle_j ^s \, dv  \\
            \leq - k_s (m_i+m_j) Y_i^{s+1} + A_s m_i m_j \left(m_i^{s-1} + m_j^{s-1} \right) Y_i^1 Y_i^s.
    \end{multline}
    We then use Jensen's inequality 
    \begin{equation}
        Y_i^{s+1} \geq (Y_i^s)^{(s+1)/s} ,
    \end{equation}
    to conclude: 
    \begin{equation*}
    	\begin{aligned}
        \frac{d}{dt} Y^s \leq - \sum_{i=1}^M &\left(2k_s m_i (Y_i^s)^{(s+1)/s} - K_{ii}^s Y_i^s \right) \\ &- \sum_{i=1}^M \sum_{\substack{j=1\\j \neq i}}^M \left(k_s (m_i+m_j) (Y_i^s)^{(s+1)/s} - K_{ij}^s Y_i^s \right),
        \end{aligned}
    \end{equation*}
    where $K_{ii}^s = 2 A_s m_i$ and $K_{ij}^s = A_s m_i\,m_j ( m_i^{s-1} + m_j^{s-1})$.
    Therefore, we have
    \begin{equation}
        \frac{d}{dt} Y^{s}(t) < 0 
    \end{equation}
    if $Y^s > \max_{i}(K_{ii}^s/(2k_s\,m_i))^s$ and $Y_s > \max_{i, j} (K_{ij}^s/(k_s\,(m_i+m_j))^s$.\\
    Finally, the upper bound holds
    \begin{equation*}
        \sup_{t>0} Y^s(t) \leq Y_*^s = \max \left\{Y^s(0), (K_{ii}^s/(2k_s\,m_i))^s, (K_{ij}^s/(k_s\,(m_i+m_j))^s \right\}.
    \end{equation*}
\end{proof}

\section{Existence and Uniqueness of Solutions}
\label{Section_MischlerMouhot}
In this section, we extend the work of Mischler, Mouhot and Ricard in \cite{mischlermouhot2006} to the case of the multi-species inelastic Boltzmann equation \eqref{homogeneous_BEs}. In particular, we will focus on the case of collision rates which do not depend on the kinetic energy of the solution. To this aim, we give estimates of the global operator in Orlicz spaces, to recover the boundedness of the evolution semi-group in any Orlicz spaces, even for not bounded bilinear collision operators. The first step consists in proving convolution-like estimates in Orlicz spaces for the gain term of a generic (\textit{intra-} or \textit{inter-}) collision operator. The second step is a generalization of the proof for the minoration of the loss term to the case of multi-species interactions, in which only the total momentum is conserved. The last step is the extension of the existence and uniqueness result in \cite{mischlermouhot2006} using the Povzner-type inequality derived in the previous Section.
\subsection{Generalization of the collision dynamics}
As observed in the Section \ref{Subsection_Micro_Dynamics}, we can obtain an explicit expression for the post-collisional velocities $\{v_*,v'_*\}$ in terms of the pre-collisional ones, starting from the conservation of the total momentum and the dissipation of the total kinetic energy. We recall the $\sigma$-representation \eqref{sigma_representation}:
\begin{equation*}
    \begin{aligned}
        v'=\frac{m_i v+ m_j v_*}{m_i+m_j} + \frac{2m_j}{m_i+m_j} \frac{(1-e)}{4} (v-v_*) + \frac{2m_j}{m_i+m_j} \frac{1+e}{4} |v-v_*| \sigma,\\
        v_*'=\frac{m_i v+ m_j v_*}{m_i+m_j} - \frac{2m_i}{m_i+m_j} \frac{(1-e)}{4} (v-v_*) - \frac{2m_i}{m_i+m_j} \frac{1+e}{4} |v-v_*| \sigma,
    \end{aligned} 
\end{equation*}
where $\sigma \in \mathbb{S}^{d-1}$.
In order to simplify the notations and to give a proof in a more general framework, we introduce the following parametrization by $z \in \mathcal{D} := \{ w \in \R^d : |w| \leq 1 \}$. The explicit expressions for the post-collisional velocities become
\begin{equation}
\label{zeta_representation}
    \begin{aligned}
        v'=\frac{m_i v+ m_j v_*}{m_i+m_j} + \frac{m_j}{m_i+m_j} z |v-v_*|,\\
        v_*'=\frac{m_i v+ m_j v_*}{m_i+m_j} - \frac{m_i}{m_i+m_j} z |v-v_*|.
    \end{aligned} 
\end{equation}
This representation is equivalent to the $\sigma$-representation when
\begin{equation}
    z = \frac{(1-e)}{2}\hat{u} + \frac{(1+e)}{2} \sigma \, \, \in \mathcal{D}.
\end{equation}
The collision rate $B$ can now be written as
\begin{equation}
\label{assumption_B_structure_z}
    B=|u|\,b(u, dz, \mathcal{E}(f)),
\end{equation}
where we assume that
\begin{equation}
\label{assumption_b_cross_section}
    b=\alpha(\mathcal{E}) \beta(\mathcal{E},u;dz),
\end{equation}
for $\beta$ a probability measure on $\mathcal{D}$ representing the normalized cross-section and $\alpha$ an intensity coefficient depending on the kinetic energy.

\subsection{Assumptions on the collision rate}
Some assumptions on the collision rate $B$ are necessary, as pointed out in \cite{mischlermouhot2006}: 
\begin{itemize}
    \item The probability measure $\beta$ satisfies the symmetry property
    \begin{equation}   \label{assumption_beta_1}
        \beta(\mathcal{E}, u ;dz) = \beta(\mathcal{E}, -u ;-dz);
    \end{equation}
    \item For any $\varphi \in C_c(\R^d)$ the functions
    \begin{equation}
    \label{assumption_beta_2}
        (v,v_*,\mathcal{E}) \to \int_D \varphi(v') \beta(\mathcal{E}, u ;dz) \quad \text{ and } \quad \mathcal{E} \to \alpha(\mathcal{E})
    \end{equation}
    are continuous on $\R^d\times \R^d\times(0,\infty)$ and $(0,\infty)$ respectively;
    \item The probability measure $\beta$ satisfies the following angular spreading property: for any $\mathcal{E} >0$, there is a function $j_\mathcal{E}(\eps) \geq 0$ such that
    \begin{equation}
    \label{assumption_beta_3}
        \forall \, \eps >0, u \in \R^d \quad \int_{|\hat{u}\cdot z | \in [-1,1] \setminus [-1+\eps, 1-\eps]} \beta(\mathcal{E}, u ;dz) \leq j_{\mathcal{E}}(\eps)
    \end{equation}
    and $j_{\mathcal{E}}(\eps) \to 0$ as $\eps \to 0$ uniformly according to $\mathcal{E}$ when it is restricted to a compact set of $(0, +\infty)$.
\end{itemize}

\subsection{Estimates in Orlicz spaces}
Following the strategy of \cite{mischlermouhot2006}, we introduce the decomposition $b=b^t+b^r$ of the cross-section $b$ as
\begin{equation}
    \begin{cases}
        b^t_\varepsilon (\mathcal{E}, u; dz) = b(\mathcal{E}, u ; dz) \mathds{1}_{\{-1+\varepsilon \leq \hat{u} \cdot z \leq 1 - \varepsilon \}}, \\
        b^r_\varepsilon (\mathcal{E}, u; dz) = b(\mathcal{E}, u ; dz) - b^t_\varepsilon (\mathcal{E}, u; dz),
    \end{cases}
\end{equation}
where $\varepsilon \in (0,1)$ and $\mathds{1}_A$ is the usual indicator function of the set $A$.\\
In order to define the Orlicz space $L^\Lambda(\R^d)$, let $\Lambda$ be a $C^2$ function strictly increasing, convex and satisfying the assumptions \eqref{Appendix_Assumption_Orlicz1}-\eqref{Appendix_Assumption_Orlicz2}-\eqref{Appendix_Assumption_Orlicz3}. 

\begin{thm}
\label{orlicz_gain_term}
    Assume that $B$ satisfies \eqref{assumption_B_structure_z}-\eqref{assumption_b_cross_section}-\eqref{assumption_beta_1}-\eqref{assumption_beta_2}-\eqref{assumption_beta_3}. For any nonzero functions $f_i, f_j \in L^1_1 \cap L^\Lambda$ and for any $\varepsilon \in (0,1)$ there exists an explicit constant $C_\mathcal{E}^+(\varepsilon)$ such that
    \begin{equation}
    \label{orlicz_gain_term_inequality}
    \begin{aligned}    
    	\int_{\R^d} Q^+(f_i, f_j) \, \Lambda' &\left( \frac{f_i}{||f_i||_{L^\Lambda}} \right) dv \leq  \alpha \left(\mathcal{E}\right) \Bigg[ C_\mathcal{E}(\varepsilon) N^{\Lambda^*} \left( \Lambda' \left( \frac{|f_i|}{||f_i||_{L^\Lambda}} \right) \right) ||f_i||_{L^1_1} ||f_j||_{L^\Lambda} \\
    	+& 2j_\mathcal{E}(\varepsilon) m_i^{-1/2} ||f_i||_{L^1_1} \int_{\R^d} f_j \Lambda' \left( \frac{f_j}{||f_j||_{L^\Lambda}} \right) |v| dv \\
    	+&2\left(\frac{m_i+m_j}{m_j}\right)^d m_i^{-1/2} ||f_i||_{L^1_1} \, \frac{||f_j||_{L^\Lambda}}{||f_i||_{L^\Lambda}} \int_{\R^d} f_i \, \Lambda' \left( \frac{f_i}{||f_i||_{L^\Lambda}} \right) |v| dv \Bigg]. 
    \end{aligned}
    \end{equation}
\end{thm}
We first state the following geometrical Lemma proven in \cite{mischlermouhot2006}, that we will use to justify the change of variables $v_* \to v'$ and $v \to v'$, keeping fixed the other variables.
\begin{lem}[\cite{mischlermouhot2006}, Lemma $2.3$]
\label{geometrical_lemma_orlicz}
    For any $z \in \mathcal{D}$ and $\gamma \in (-1,1)$, we define the map
    \begin{equation}
        \Phi_z : \R^d \to \R^d, \qquad \R^d \ni u \to w = \Phi_z(u):=u+|u|z,
    \end{equation}
    its Jacobian function $J_z := \det(D\Phi_z)$ and the cone $\Omega_\gamma := \{ u \in \R^d \setminus \{0\}, \hat{u} \cdot \hat{z} > \gamma \}$.\\
    Then $\Phi_z$ is a $C^\infty$-diffeomorphism from $\Omega_\gamma$ onto $\Omega_\delta$, where
    \begin{equation*}
        \delta= \frac{\gamma+|z|}{(1+2\gamma|z|+|z|^2)^{\frac{1}{2}}}
    \end{equation*}
    and there exists $C_\gamma \in (0, \infty)$ such that
    \begin{equation*}
        C_\gamma^{-1} \leq J_z \leq C_\gamma \qquad \text{ on } \,\, \Omega_\gamma,
    \end{equation*}
    uniformly with respect to $z \in \mathcal{D}$.
\end{lem}

\begin{proof}[Proof of Theorem \ref{orlicz_gain_term}]
    Let us consider for $f_i \neq 0$
    \begin{equation}
        \varphi_i(f_i)=\Lambda'\left(\frac{f_i}{||f_i||_{L^\Lambda}} \right)
    \end{equation}
    Using the decomposition $b=b^t+b^r$, we can rewrite
    \begin{equation}
    \begin{aligned}
        \int_{\R^d} Q^+(f_i,f_j) \varphi(f_i)dv = & \int_{\R^d \times \R^d \times \mathcal{D}} f_i f_{j*} \varphi_i(f_i') |u| b^t (\mathcal{E}, u; dz) \, dv\,dv_* \\
        &+ \int_{\R^d \times \R^d \times \mathcal{D}} f_i f_{j*} \varphi_i(f_i') |u| b^r (\mathcal{E}, u; dz) \, dv\,dv_* =: I^t + I^r.
    \end{aligned}      
    \end{equation}
    Using the bound $|u| \leq |v| + |v_*|$, we have
    \begin{equation}
        \begin{aligned}
            I^t \leq & \int_{\R^d \times \R^d \times \mathcal{D}} f_i f_{j*} \varphi(f_i') |v| b^t (\mathcal{E}, u; dz) \, dv\,dv_*\\
            &+\int_{\R^d \times \R^d \times \mathcal{D}} f_i f_{j*} \varphi(f_i') |v_*| b^t (\mathcal{E}, u; dz) \, dv\,dv_* =: I_1^t + I_2^t
        \end{aligned}
    \end{equation}
    For the term $I_1^t$ we apply Young's inequality: $xy \leq \Lambda(x) + \Lambda^*(y)$. Hence,
    \begin{equation*}
        f_{j*} \varphi(f_i') = ||f_{j}||_{L^\Lambda} \left( \frac{f_{j*}}{||f_{j}||_{L^\Lambda}} \right) \varphi(f_i') \leq ||f_{j}||_{L^\Lambda} \Lambda \left( \frac{f_{j*}}{||f_{j}||_{L^\Lambda}} \right) + ||f_{j}||_{L^\Lambda} \Lambda^*(\varphi(f_i'))
    \end{equation*}
    and we get
    \begin{equation}
    \begin{aligned}
        I_1^t \leq &\frac{1}{m_i^{1/2}} ||f_{j}||_{L^\Lambda} \int_{\R^d \times \R^d \times \mathcal{D}} f_i m_i^{1/2}|v| \Lambda \left(\frac{f_{j*}}{||f_j||_{L^\Lambda}} \right) b^t (\mathcal{E}, u; dz) \, dv\,dv_* \\
        + &\frac{1}{m_i^{1/2}} ||f_{j}||_{L^\Lambda} \int_{\R^d \times \R^d \times \mathcal{D}} f_i m_i^{1/2}|v| \Lambda^* \left(\varphi_i (f_i') \right) b^t (\mathcal{E}, u; dz) \, dv\,dv_* =: I^t_{1,1} + I^t_{1,2}.
    \end{aligned}       
    \end{equation}
    For the first term $I^t_{1,1}$ we can use the inequality
    \begin{equation}
        \Lambda(x) \leq x \Lambda'(x), \quad \forall x \in \R_+,
    \end{equation}
    obtaining
    \begin{equation}
        I_{1,1}^t \leq \frac{\alpha(\mathcal{E})}{m_i^{1/2}} ||f_i||_{L^1_1} \int_{\R^d} f_j \varphi_i(f_j) \, dv.
    \end{equation}
    H\"older's inequality in Orlicz spaces \eqref{Appendix_Orlicz} yields 
    \begin{equation}
        I^t_{1,1} \leq \frac{\alpha(\mathcal{E})}{m_i^{1/2}} N^{\Lambda^*} \left(\Lambda'\left(\frac{|f_j|}{||f_j||_{L^\Lambda}} \right) \right) ||f_i||_{L^1_1} ||f_j||_{L^\Lambda},
    \end{equation}
    where we recall that $N^{\Lambda^*}$ is the norm on the Orlicz space associated to the complementary function $L^*$, defined in \eqref{norm_N^*}.\\
    For the term $I_{1,2}^t$ we use that $\Lambda^*(y)=y \left(\Lambda'\right)^{-1}(y) - \Lambda \left(\left(\Lambda'\right)^{-1}(y)\right)$ to get
    \begin{equation}
        I_{1,2}^t \leq \frac{1}{m_i^{1/2}} \frac{||f_j||_{L^\Lambda}}{||f_i||_{L^\Lambda}} \int_{\R^d \times \R^d \times \mathcal{D}} m_i^{1/2} f_i |v| \varphi(f_i') f_i' b^t(\mathcal{E}, u; dz) \, dv\,dv_*.
    \end{equation}
    Let us introduce the change of variables $\Psi : (v,v_*,z) \to (v, \psi_{v,z}(v_*),z)$ with $\psi_{v,z}(v_*) = v'=v+\displaystyle \frac{m_j}{m_i+m_j} \Phi_z(v-v_*)$. Since we have decomposed the cross section $b$, here we are truncating the integration in the region $-1+\eps \leq \hat{u} \cdot z \leq 1-\eps$. Thanks to Lemma \ref{geometrical_lemma_orlicz} the application $\Psi$ is a $C^\infty$-diffeomorphism from 
    $\{(v,v_*,z) \in \R^d \times \R^d \times \mathcal{D}, \hat{u} \cdot z \neq 1 \}$ onto its image. The jacobian of the transformation is 
    \begin{equation}
        J_\Psi=\displaystyle \left(\frac{m_j}{m_i+m_j} \right)^d (1 - \hat{u} \cdot z)
    \end{equation} 
    and it satisfies $|J_\Psi^{-1}| \leq \displaystyle \left(\frac{m_i+m_j}{m_j} \right)^d \varepsilon^{-1}$. We finally get
    \begin{equation}
        \begin{aligned}
        I_{1,2}^t \leq &\frac{1}{m_i^{1/2}} \frac{||f_j||_{L^\Lambda}}{||f_i||_{L^\Lambda}} \int_{\R^d \times \R^d \times \mathcal{D}} m_i^{1/2} f_i |v| \varphi(f_i') f_i' J_\Psi^{-1} b^t(\mathcal{E}, v-\psi^{-1}_{v,z}(v'); dz) dv\,dv' \\
        \leq & \alpha(\mathcal{E}) \left(\frac{m_i+m_j}{m_j} \right)^d m_i^{-1/2} \varepsilon^{-1} \frac{||f_j||_{L^\Lambda}}{||f_i||_{L^\Lambda}} ||f_i||_{L^1_1} \int_{\R^d} f_i \varphi(f_i) dv.
        \end{aligned}
    \end{equation}
    Using again H\"older's inequality, we obtain
    \begin{equation}
        I_{1,2}^t \leq \alpha(\mathcal{E}) \left(\frac{m_i+m_j}{m_j} \right)^d m_i^{-1/2} \varepsilon^{-1} ||f_j||_{L^\Lambda} ||f_i||_{L^1_1} \, N^{\Lambda^*} \left( \Lambda' \left( \frac{|f_i|}{||f_i||_{L^\Lambda}} \right) \right). 
    \end{equation}
    Since the expression for $I_2^t$ is similar to $I_1^t$, we get
    \begin{equation}
    \label{inequality_It_orlicz}
        I^t \leq 2 \alpha(\mathcal{E}) m_i^{-1/2} \left(1 + \left(\frac{m_i+m_j}{m_j} \right)^d \varepsilon^{-1}\right) ||f_i||_{L^1_1} \, N^{\Lambda^*} \left( \Lambda' \left( \frac{|f_i|}{||f_i||_{L^\Lambda}} \right) \right) ||f_j||_{L^\Lambda}
    \end{equation}

    Now we focus on the term $I^r$, splitting it as
    \begin{equation}
        \label{ineq:I1rI2r}
        \begin{aligned}
        I^r \leq &\int_{\R^d \times \R^d \times \mathcal{D}} f_i f_{j*} \varphi(f_i') \mathds{1}_{\{\hat{u} \cdot z \leq 0\}} |u| b^r(\mathcal{E}, u; dz) dv\,dv_* \\
        + & \int_{\R^d \times \R^d \times \mathcal{D}} f_i f_{j*} \varphi(f_i') \mathds{1}_{\{\hat{u} \cdot z > 0\} } |u| b^r(\mathcal{E}, u; dz) dv\,dv_* =: I^r_1 + I^r_2.
        \end{aligned}
    \end{equation}
    For the term $I_1^r$ we use Young's inequality on $x=f_{j*}$ and $y=\varphi(f_i')$ and \eqref{ineq:I1rI2r} to obtain
    \begin{equation}
    \begin{aligned}
      I_1^r \leq & \int_{\R^d \times \R^d \times \mathcal{D}} f_i f_{j*} \varphi(f_{j*})  \mathds{1}_{\{\hat{u} \cdot z \leq 0\}} |u| b^r(\mathcal{E}, u; dz) \, dv\,dv_* \\
      + &\frac{||f_j||_{L^\Lambda}}{||f_i||_{L^\Lambda}} \int_{\R^d \times \R^d \times \mathcal{D}} f_i f_i' \varphi(f_i')  \mathds{1}_{\{\hat{u} \cdot z \leq 0\}} |u| b^r(\mathcal{E}, u; dz) \, dv\,dv_*.
    \end{aligned}    \end{equation}
    Using again the change of variable defined by $\Psi$ in the second integral for which $|J_{\Psi}^{-1}| \leq \displaystyle \left(\frac{m_i+m_j}{m_j} \right)^d$ on the domain of integration, we have
    \begin{equation}
    \begin{aligned}
        I_1^r \leq & \, \left( \sup_{u \in \R^d} \int_D b^r ( \mathcal{E}, u; dz) \right) ||f_i||_{L^1_1} m_i^{-1/2} \int_{\R^d} f_{j} \varphi(f_j) (1+|v|) \, dv \\ 
        + & 2 \left(\frac{m_i+m_j}{m_j} \right)^d \left( \sup_{u \in \R^d} \int_D b^r ( \mathcal{E}, u; dz) \right) \frac{||f_j||_{L^\Lambda}}{||f_i||_{L^\Lambda}} ||f_i||_{L_1^1} m_i^{-1/2} \int_{\R^d} f_{i} \varphi(f_i) (1+|v|) \, dv \\
        \leq & \,  \alpha(\mathcal{E})j_\mathcal{E}(\eps) ||f_i||_{L^1_1} m_i^{-1/2} \int_{\R^d} f_{j} \varphi(f_j) (1+|v|) \, dv \\ 
        + & 2\left(\frac{m_i+m_j}{m_j} \right)^d \alpha(\mathcal{E})j_\mathcal{E}(\eps) \frac{||f_j||_{L^\Lambda}}{||f_i||_{L^\Lambda}} ||f_i||_{L_1^1} m_i^{-1/2} \int_{\R^d} f_{i} \varphi(f_i) (1+|v|) \, dv
        \end{aligned}
    \end{equation}
    We treat in the same way the term $I_2^r$ and we obtain the estimate
    \begin{equation}
    \label{inequality_Ir_orlicz}
        \begin{aligned}
        I^r \leq &2\alpha(\mathcal{E})j_\mathcal{E}(\eps) ||f_i||_{L^1_1} m_i^{-1/2} \int_{\R^d} f_{j} \varphi(f_j) (1+|v|) dv + \\ 
        + & 4\left(\frac{m_i+m_j}{m_j} \right)^d \alpha(\mathcal{E})j_\mathcal{E}(\eps) \frac{||f_j||_{L^\Lambda}}{||f_i||_{L^\Lambda}} ||f_i||_{L_1^1} m_i^{-1/2} \int_{\R^d} f_{i} \varphi(f_i) (1+|v|) dv
        \end{aligned}
    \end{equation}
    Applying again the H\"older's inequality in Orlicz spaces and defining 
    \begin{equation}
        C^+_\mathcal{E}(\varepsilon) = 2 \left(1 + \left(\frac{m_i+m_j}{m_j} \right)^d \varepsilon^{-1}\right) + 2 \left(1 + 2\left(\frac{m_i+m_j}{m_j} \right)^d \right) j_\mathcal{E}(\varepsilon),
    \end{equation}
    we conclude the proof gathering \eqref{inequality_It_orlicz} and \eqref{inequality_Ir_orlicz}.
\end{proof} 

\subsection{Minoration of the loss term}
In this subsection we adapt a result about the minoration of the loss term to the case of the multi-species inelastic Boltzmann equation. We restrict to the case of two species, without loss of generality.
\begin{lem}
\label{lemma_jensen_inequality}
    For any non-negative measurable functions $f_i$, $f_j$ such that
    \begin{equation}
    \label{hyp_minoration}
    	\begin{aligned}
        f_i, f_j \in L^1_1(\R^d), \qquad &\int_{\R^d} m_i f_i\,dv=1, \qquad \int_{\R^d} m_j f_j \, dv=1, \\ &\int_{\R^d} (m_i f_i + m_j f_j) \, v \, dv=0,
        \end{aligned}
    \end{equation}
    we have
    \begin{equation}
        \int_{\R^d} \left(m_i f_{i*} + m_j f_{j*} \right) |v-v_*|^{\lambda+2} \, dv_* \geq |v|^{\lambda + 2}, \qquad \forall v \in \R^d.
    \end{equation}
\end{lem}
\begin{proof}
    Use Jensen's inequality 
    \begin{equation}
        \int_{\R^d} \varphi(g_*) \,d\mu_* \geq \varphi \left(\int_{\R^d} g_* \, d\mu_* \right)
    \end{equation}
    with the probability measure $d\mu_*=\displaystyle \frac{(m_i f_{i*}+m_j f_{j*})}{2}dv_*$, the measurable function $v_* \to g=v-v_*$ and the convex function $\varphi(s)=|s|^{\lambda+2}$. Hence, we obtain
    \begin{equation*}
    \int_{\R^d} [m_i\,f_{i*} + m_j\,f_{j*}] |v-v_*|^{\lambda+2} \, dv_* \geq \Big|v-\int_{\R^d} [m_i\,f_{i*} v_* + m_j \, f_{j*} v_*] \, dv_* \Big|^{\lambda+2} = |v|^{\lambda+2} 
\end{equation*}
\end{proof}
At this point, we can prove the following proposition
\begin{prop}
    Assume that $B$ satisfies \eqref{assumption_B_structure_z}-\eqref{assumption_b_cross_section}-\eqref{assumption_beta_1}-\eqref{assumption_beta_2}-\eqref{assumption_beta_3}. For any non-negative functions $f_i$, $f_j$ satisfying \eqref{hyp_minoration}, we have
    \begin{equation}
    \label{orlicz_loss_term_inequality}
        \int_{\R^d} Q^-(f_i,f_j) \Lambda' \left(\frac{f_i}{||f_i||_{L^\Lambda}} \right) \, dv \geq \frac{1}{m_i+m_j} \alpha(\mathcal{E}) \int_{\R^d} f_i \Lambda' \left( \frac{f_i}{||f_i||_{L^\Lambda}} \right) |v| \, dv. 
    \end{equation}
\end{prop}
\begin{proof}
    Let us consider 
    \begin{equation*}
    	\begin{aligned}
        \int_{\R^d} Q^-(f_i,f_j) \Lambda' &\left(\frac{f_i}{||f_i||_{L^\Lambda}} \right) dv =\\ &=\int_{\R^d} Q_{i,i}^-(f_i,f_i) \Lambda' \left(\frac{f_i}{||f_i||_{L^\Lambda}} \right) \,dv + \int_{\R^d} Q_{i,j}^-(f_i,f_j) \Lambda' \left(\frac{f_i}{||f_i||_{L^\Lambda}} \right)\, dv.
        \end{aligned}
    \end{equation*}
    We can rewrite the second term as
    \begin{equation*}
        \begin{aligned}
        \frac{1}{m_i} \int_{\R^d} m_i Q_{i,i}^-(f_i,f_i) &\Lambda' \left(\frac{f_i}{||f_i||_{L^\Lambda}} \right) \, dv + \frac{1}{m_j}\int_{\R^d} m_j Q_{i,j}^-(f_i,f_j) \Lambda' \left(\frac{f_i}{||f_i||_{L^\Lambda}} \right) \, dv \geq \\
        \geq & \, \alpha(\mathcal{E}) \frac{1}{m_i+m_j} \int_{\R^d \times \R^d} f_i (m_i f_{i*} + m_j f_{j*})|v-v_*| \Lambda' \left(\frac{f_i}{||f_i||_{L^\Lambda}} \right) \, dv\,dv_*. 
        \end{aligned}
    \end{equation*}
    We conclude the proof thanks to the Lemma \ref{lemma_jensen_inequality}
\end{proof}

\subsection{Estimation on the solutions}
\begin{thm}
\label{thm_estimation_Q}
    Assume that B satisfies \eqref{assumption_B_structure_z}-\eqref{assumption_b_cross_section}-\eqref{assumption_beta_1}-\eqref{assumption_beta_2}-\eqref{assumption_beta_3} and let us consider $f_i$, $f_j$ non-negative functions satisfying \eqref{hyp_minoration}. Then we can define an explicit constant $C_\mathcal{E}$ depending on the collision rate through $\alpha$ and $j_\mathcal{E}$ such that
    \begin{equation}
        \int_{R^d} \mathcal{Q} (f_i,f_j) \Lambda' \left(\frac{f_i}{||f_i||_{L^\Lambda}} \right) dv \leq C_\mathcal{E} \left[ N^{\Lambda^*} \left( \Lambda' \left(\frac{|f_i|}{{||f_i||}_{L^\Lambda}} \right) \right) \right]^{-1} ||f_i||_{L^1_1} ||f_j||_{L^\Lambda}.
    \end{equation}
\end{thm}
\begin{proof}
    We can straightforwardly conclude combining \eqref{orlicz_gain_term_inequality} and \eqref{orlicz_loss_term_inequality}, choosing a small $\eps_0$ such that
    \begin{equation*}
        2 \left(1+ \left(\frac{m_i+m_j}{m_j} \right)^d \right) j_\mathcal{E}(\varepsilon_0) ||f_i||_{L^1_1} \leq 1.
    \end{equation*}
\end{proof}
\begin{lem}
\label{corollary_estimation}
    Assume that $B$ satisfies \eqref{assumption_B_structure_z}-\eqref{assumption_b_cross_section}-\eqref{assumption_beta_1}-\eqref{assumption_beta_2}-\eqref{assumption_beta_3} and let us consider $f_i, f_j \in C([0,T]; L^1_2)$ solutions to the Boltzmann equation \eqref{homogeneous_BEs} associated to an initial condition $f_i^{in}, f_j^{in} \in L^1_2$ and to the collision rate $B$. Assume that 
    \begin{equation}
        \int_{{\R^d}} m_i f_i\,dv=1, \qquad \int_{\R^d} m_j f_j\,dv=1, \qquad \int_{\R^d} (m_i f_i + m_j f_j) v\,dv=0,
    \end{equation}
    and there exists a compact set $K \subset (0, +\infty)$ such that
    \begin{equation}
        \mathcal{E}(f)(t) \in K \quad \forall \, t \in [0,T].
    \end{equation}
    Then, there exists a $C^2$, strictly increasing and convex function $\Lambda$ satisfying the assumptions \eqref{Appendix_Assumption_Orlicz1}-\eqref{Appendix_Assumption_Orlicz2}-\eqref{Appendix_Assumption_Orlicz3} and a constant $C_T$ such that
    \begin{equation}
        \sup_{[0,T]} ||f(t, \cdot)||_{L^\Lambda} \leq C_T,
    \end{equation}
    where $C_T$ depends only on $K$, $T$ and $B$.
\end{lem}
\begin{proof}
    Since $f_i \in L^1(\R^d) \,\, \forall \, i$, the refined version of the De la Vallée-Poussin Theorem recalled in the Appendix \ref{Appendix_Orlicz} guarantees that there exists a function $\Lambda$ satisfying the properties of Corollary \ref{corollary_estimation} and such that 
    \begin{equation*}
        \int_{\R^d} \Lambda(|f_{in}|) \, dv < \infty.
    \end{equation*}
    Then, thanks to Theorem \ref{theorem_differentiation_orlicz}, the $L^\Lambda$ norm of $f_i$ satisfies
    \begin{equation}
        \frac{d}{dt} ||f_i||_{L^\Lambda}= \sum_{i,\, j} \left[ N^{\Lambda^*} \left(\Lambda’ \left(\frac{|f_i|}{||f_i||_{L^\Lambda}} \right) \right) \right]^{-1} \int_{\R^d} \mathcal{Q}(f_i,f_j) \, \Lambda’ \left(\frac{|f_i|}{||f_i||_{L^\Lambda}} \right) \, dv.
    \end{equation}
    Using Theorem \ref{thm_estimation_Q} one can bound the integral of the collision operator and
    \begin{equation}
        \frac{d}{dt} ||f_i||_{L^\Lambda} \leq \sum_{i,\, j} C_{\mathcal{E}(f)(t)} ||f_i||_{L^1_1} ||f_j||_{L^\Lambda}, \qquad \forall \, i, \quad \forall \, t \in [0,T].
    \end{equation}
    The constant $C_{\mathcal{E}(f)(t)}$ is uniformly bound when the kinetic energy belongs to a compact set, thus we obtain the following inequality
    \begin{equation}
        \frac{d}{dt} \sum_{i,\, j} ||f_i||_{L^\Lambda} \leq C_K ||f||_{L^1_1} \sum_{i,\, j}  ||f_j||_{L^\Lambda} \qquad \forall \, t \in [0,T]
    \end{equation}
    for a positive constant $C_K$ depending on $K$ and the collision rate. The end of the proof is straightforward by using a Gronwall's argument.
\end{proof}

\subsection{Cauchy Theory for non-coupled collision rate}
As in \cite{mischlermouhot2006}, we fix $T_* > 0$ and we consider the following assumptions on the collision rate:
\begin{align}
	\label{B_assumption1_Cauchy}
	&B=B(t,u;dz)=|u| \gamma(t) b(t,u;dz),\\
	\label{B_assumption2_Cauchy}
	&b(t,u;dz)=b(t,-u;-dz) \qquad \forall \, t \in [0,T_*], \,\, \forall\,u \in \R^d,\\
	\label{B_assumption3_Cauchy}
	&0 \leq \gamma(t) \leq \gamma_* \quad \text{in } \, (0,T_*).
\end{align}
Let us first show the uniqueness of the solution to the multi-species inelastic Boltzmann equation.
\begin{thm}[Existence and Uniqueness]
\label{existence_thm}
    Assume that $B$ satisfies \eqref{B_assumption1_Cauchy}-\eqref{B_assumption2_Cauchy}-\eqref{B_assumption3_Cauchy} with $b=b(u;dz)$, namely the cross section does not depend on the kinetic energy. Take as initial data $f_i^{in}$  satisfying \eqref{hyp_minoration} with $q=3$. Then for all $T>0$, there exists a unique solution $\f = (f_i)_{i=1}^M \in \left(C([0,T]; L_2^1) \cap L^\infty(0,T; L^1_3) \right)^{\otimes M}$ to the multi-species inelastic Boltzmann equation \eqref{homogeneous_BEs}. The solution $f$ conserves single-species mass and total momentum 
        \begin{equation}
        	\label{conservation_properties_fi}
            \int_{\R^d} m_i f_i(t,v)dv=1, \qquad \sum_{i=1}^M\int_{\R^d} m_i f_i(t,v)\, v \, dv = 0, \qquad \forall t \in [0,T]
        \end{equation}
    and it has a positive decreasing total kinetic energy.
\end{thm}
The proof of the uniqueness of solution $\f$ to the multi-species inelastic Boltzmann equation relies in the following Proposition, where we show the stability estimate in $L^1_2$. 
\begin{prop}
\label{uniqueness_prop}
    Assume that $B$ satisfies \eqref{B_assumption1_Cauchy}-\eqref{B_assumption2_Cauchy}-\eqref{B_assumption3_Cauchy}. For any solutions $(f_i)$ and $(g_i)$ of the inelastic Boltzmann equation on $[0,T]$ with $T \leq T_*$ we have
    \begin{equation}
    \label{differential_inequality_uniqueness}
        \displaystyle \frac{d}{dt} \sum_{i=1}^M \int_{\R^d} |f_i-g_i|(1+m_i\,|v|^2)\, dv \leq C \gamma_* \int_{\R^d} (f+g)\,(1+|v|^2)^{3/2} dv \int_{\R^d} |f-g| (1+|v|^2) dv 
    \end{equation}
\end{prop}
\begin{cor}
    From the inequality \eqref{differential_inequality_uniqueness}, we deduce that there exists a constant $C_T>0$ depending on $B$ and $\displaystyle \sup_{t \in [0,T]} ||f+g||_{L^1_3}$ such that
    \begin{equation}
        \label{exponential_inequality_uniqueness}
        ||f-g||_{L^1_2} \leq ||f^{in}-g^{in}||_{L^1_2} e^{C_T t}.
    \end{equation}
    In particular, this inequality guarantees the uniqueness of the solution to the Cauchy problem for the multi-species inelastic Boltzmann equation in $define space$.
\end{cor}

\begin{proof}
We multiply the equation satisfied by $(f_i-g_i)$ by $\phi_i(t,y)=\text{sgn}(f_i(t,y)-g_i(t,y))\,k_i$, where $k_i=(1+m_i|v|^2)$. Using equation \eqref{homogeneous_BEs} and the weak formulation of the collision operator in \eqref{weak_monospecies} and \eqref{weak_bispecies}, we obtain
\begingroup
    \allowdisplaybreaks
    \begin{align*}
    \frac{d}{dt} \sum_{i=1}^M \int_{\R^d} |f_i-g_i| k_i dv = \phantom{+} \frac{1}{2} \sum_{i=1}^M \int_{\R^d \times \R^d \times D} &[(f_i-g_i)g_{i*} + f_i(f_{i*}-g_{i*})]\\ &(\phi_i' + \phi'_{i*} - \phi_i - \phi_{i*}) B(t,u;dz) \, dv_*\, dv \\ + \frac{1}{2} \sum_{\substack{i, j=1 \\ i \neq j}}^M \int_{\R^d \times \R^d \times D} &[(f_i-g_i)g_{j*} + f_i(f_{j*}-g_{j*})]\\
    &(\phi_i' + \phi_{j*}'-\phi_i-\phi_{j*}) B(t,u;dz) \, dv_*\,dv \\
    \leq \phantom{+} \frac{1}{2} \sum_{i=1}^M \int_{\R^d \times \R^d
    \times D} & |f_i-g_i|(f_{i*}+g_{i*})]\\ &(k_i' + k'_{i*} - k_i + k_{i*}) B(t,u;dz) \, dv_*\, dv \\
    + \frac{1}{2} \sum_{\substack{i, j=1 \\ i \neq j}}^M \int_{\R^d \times \R^d \times D} &|f_i-g_i|g_{j*}\\
    &(k_i' + k_{j*}'-k_i+k_{j*}) B(t,u;dz) \, dv_*\,dv \\
    + \frac{1}{2} \sum_{\substack{i, j=1 \\ i \neq j}}^M \int_{\R^d \times \R^d \times D} &|f_{j}-g_{j}|f_{i*}\\
    &(k_{i*}' + k_{j}'+k_{i*}-k_{j}) B(t,u;dz) \, dv_*\,dv \\
    = \phantom{+} \frac{1}{2} \sum_{i=1}^M \int_{\R^d \times \R^d
    \times D} & |f_i-g_i|(f_{i*}+g_{i*})\\ &(k_i' + k'_{i*} - k_i + k_{i*}) B(t,u;dz) \, dv_*\, dv\\
    + \frac{1}{2} \sum_{i=1}^{M-1} \sum_{j=i+1}^M \int_{\R^d \times \R^d \times D} &|f_i-g_i|(f_{j*} + g_{j*})\\
    &(k_i' + k_{j*}'-k_i+k_{j*}) B(t,u;dz) \, dv_*\,dv \\
    + \frac{1}{2} \sum_{i=1}^{M-1} \sum_{j=i+1}^M \int_{\R^d \times \R^d \times D} &|f_{j}-g_{j}|(f_{i*}+g_{i*})\\
    &(k_{i*}' + k_{j}'+k_{i*}-k_{j}) B(t,u;dz) \, dv_*\,dv,
    \end{align*}
\endgroup
where we have used the symmetry hypothesis on $B$, the change of variable $(v,v_*) \to (v_*,v)$ and the structure of the functions $\phi_i$ and $\phi_j$. Then, thanks to the bounds \eqref{conservation_properties_fi} and the relations
\begin{flalign*}
    k_i'+&k'_{i*}-k_i+k_{i*} =\\ & (1+m_i|v'|^2) + (1+m_i|v'_*|^2) - (1-m_i|v|^2) + (1+m_i|v_*|^2) \leq 2 + 2m_i|v_*|^2 = 2k_{i*},
\end{flalign*} 
\begin{flalign*}
    k_i'+k'_{j*}-k_i+k_{j*} \leq 2 + 2m_j|v_*|^2 = 2k_{j*},
\end{flalign*}
\begin{flalign*}
    k_{i*}'+k'_{j}+k_{i*}-k_j \leq 2 + 2m_j|v|^2 = 2k_{i*},
\end{flalign*}
we can deduce that
\begin{align*}
    \frac{d}{dt} \sum_{i=1}^M \int_{\R^d} |f_i-g_i|k_i dv \leq &\sum_{i=1}^M \gamma_* \int_{\R^d \times \R^d} |u| |f_i-g_i| (f_{i*}+g_{i*})k_{i*} \, dv_*\,dv \\
    + &\sum_{i=1}^{M-1} \sum_{j=i+1}^M \gamma_* \int_{\R^d \times \R^d} |u| |f_i-g_i| (f_{j*}+g_{j*})k_{j*} \, dv_*\,dv \\
    + &\sum_{i=1}^{M-1} \sum_{j=1+1}^M \gamma_* \int_{\R^d \times \R^d} |u| |f_j-g_j| (f_{i*}+g_{i*})k_{i*} \, dv_*\,dv \\
    = &\sum_{i=1}^{M} \sum_{j=1}^M \gamma_* \int_{\R^d \times \R^d} |u| |f_i-g_i| (f_{j*}+g_{j*})k_{j*} \, dv_*\,dv.
\end{align*}
Now we have 
\begin{align*}
    |u|^2= &|v| + |v_*|^2 - 2 \langle v,v_* \rangle \\
    \leq & |v|^2 + |v_*|^2 + 2 |v||v_*|  \\
    \leq & (1+|v|^2) (1+|v_*|^2) \\
    \leq & \, C \, (1+m_i |v|^2) (1+m_j|v_*|^2),
\end{align*}
and we obtain 
\begin{align*}
    \frac{d}{dt} \sum_{i=1}^{M} &\int_{\R^d} |f_i-g_i| k_i \, dv  \\ & \leq C \gamma_* \sum_{i=1}^M \sum_{j=1}^{M} 
    \left( \int_{\R^d \times \R^d}  |f_i-g_i|k_i \, dv \right) \, \left( \int_{\R^d \times \R^d} (f_{j*}+g_{j*}) k_{j*}^{3/2} dv_*\right) \\[7pt]
    &= C \gamma_* \displaystyle \left(\sum_{i=1}^{M} 
    \int_{\R^d \times \R^d}  |f_i-g_i|k_i dv \right) \, \displaystyle \left( \sum_{j=1}^M \int_{\R^d \times \R^d} (f_{j*}+g_{j*}) k_{j*}^{3/2} dv_*\right).
\end{align*}
which yields the differential inequality \eqref{differential_inequality_uniqueness}. We can finally conclude by using a Gronwall argument to obtain \eqref{exponential_inequality_uniqueness}.
\end{proof}

Now we are ready to prove the existence of the solution, following the strategy of \cite{mischlermouhot2006}. The main difference is the weighted Banach space used for the estimation, which allows us to exploit the global conservation properties of the mixture. 

\begin{proof}[Proof of Theorem \ref{existence_thm}]
    The main idea is to exploit the fact that the solution is uniformly bounded in an Orlicz space, as a consequence of the result of the propagation of Orlicz norm for the collision operator.\\ 
    Let us consider initial data $(f_i^{in})$ satisfying \eqref{hyp_minoration} with $q=4$ and let us introduce the truncated collision rates $B_n=B \mathds{1}_{|u| \leq n}$. The associate collision operators $Q_n$ are bounded in $L^1_q$ for any $q \geq 1$ and they are Lipschitz in $L^1_2$ on any bounded subset of $L^1_2$. Using the Banach fixed point Theorem, we can obtain the existence of non-negative solutions $f_i^n \in C([0,T]: L^1_2) \cap L^\infty(0,T;L^1_4)$ for any $T>0$. The propagation of Orlicz norm implies that the solutions $f^n$ are uniformly bounded for $t \in [0,T]$ in a certain Orlicz space. The boundness and the propagation of moments implies the uniformly integrability of $f$ in $L^1$ (see \ref{de_la_vallee_poussin} in the Appendix), which is equivalent to the weakly compactness property by Dunford-Pettis criterion (see \ref{dunford_pettis} in the Appendix). Therefore, the sequence $(f^n)$ is weakly compact in $L^1$ so there exists a function $f$ in the same space such that $f^n \to f$ up to the extraction of a subsequence.
    \end{proof}

\section{Entropy-based proof of Haff's Law for inelastic mixtures of hard-spheres}
\label{Section_Haffs_Law}
Based on the work of Alonso and Lods \cite{alonso2013, alonso2010}, we would like to derive a generalized Haff's law for mixtures of granular perfectly smooth hard-spheres particles, without the use of self-similar variables. We underline that this proof, based on the use of Boltzmann's entropy, asserts the validity of Haff's law for initial data of finite entropy. In particular, we would like to show that
\begin{equation}
	\label{Haffs_Law}
	c(1+t)^{-2} \leq \mathcal{E}(t) \leq C(1+t)^{-2}, \qquad t \geq 0,
\end{equation}
where $c$ and $C$ are two positive constants and $\mathcal E(t)$ stands for $\mathcal{E}(f)(t)$ for the sake of simplicity.\\
The upper bound is provided by the energy dissipation inequality \eqref{energy_dissipation}, while the lower bound will be obtained using the entropy inequality.\\
Let us define the general k$^\text{th}$ order moment for a bi-species mixture as
\begin{equation}
	\label{moments_expression_mu}
	\mu_k(f) = \int_{\R^d} \left(m_1^{\frac{k}{2}}\, f_1(v) + m_2^{\frac{k}{2}}\,f_2(v)\right) |v|^k \, dv.
\end{equation}
Let us recall the expression for the entropy production functional \eqref{entropy_functional}
\begin{equation*}
	\frac{d}{dt} \mathcal{H}(f(t)) = - \mathcal{D}_e(f(t)) + \mathcal{N}_e(f(t)),
\end{equation*}
where we have defined as $-\mathcal{D}_e$ the negative quantity representing the dissipative part of the entropy production
\begin{equation*}
	\begin{aligned}
		\mathcal{D}_e(f) := &- \frac{1}{2}\int_{\R^d \times \R^d \times \mathbb{S}^{d-1}} B f_{1*} f_1 \left[\log \left(\frac{f_1' f_{1*}'}{f_1 f_{1*}} \right) - \frac{f_1' f_{1*}'}{f_1 f_{1*}} + 1 \right] d\sigma \, dv \, dv_*\\
		&- \phantom{\frac{1}{2}} \int_{\R^d \times \R^d \times \mathbb{S}^{d-1}} B f_1 f_{2*}  \left[\log \left(\frac{f_1' f_{2*}'}{f_1 f_{2*}} \right) - \frac{f_1' f_{2*}'}{f_1 f_{2*}} + 1 \right] d\sigma \, dv \, dv_*\\
		&- \frac{1}{2}\int_{\R^d \times \R^d \times \mathbb{S}^{d-1}} B f_{2*} f_2 \left[\log  \left(\frac{f_2' f_{2*}'}{f_2 f_{2*}} \right) - \frac{f_2' f_{2*}'}{f_2 f_{2*}} + 1 \right] d\sigma \, dv \, dv_*,
	\end{aligned}
\end{equation*}
and as $\mathcal{N}_e$ the other term given by
\begin{equation*}
	\begin{aligned}
		\mathcal{N}_e(f) :&= \frac{1}{2} \int_{\R^d \times \R^d \times \mathbb{S}^{d-1}} B \left(f_{1*}' f_1' - f_{1*} f_1 + f'_{2*} f'_2 - f_{2*} f_2 \right) d\sigma \, dv \, dv_*\\
		&+\frac{1}{2} \int_{\R^d \times \R^d \times \mathbb{S}^{d-1}} B \left(f_{2*}' f_1' - f_{2*} f_1 \right) d\sigma \, dv \, dv_*.
	\end{aligned}
\end{equation*}
In the framework of constant restitution coefficients, the Jacobian for the transformation $(v',v'_*,\sigma) \to (v,v_*,\sigma)$ is independent from the relative velocity, thus we can explicitly rewrite the expression for $\mathcal{N}_e$ as
\begin{equation*}
	\begin{aligned}
		\mathcal{N}_e(f) = \frac{1-e^2}{e^2} \int_{\R^d \times \R^d}  |u| [f_{1*} f_1 + 2 f_{2*}\,f_1 + f_{2*} f_2] \, dv \, dv_*.
	\end{aligned}
\end{equation*}
The main idea is to prove the logarithmic growth of the entropy, in order to derive the lower bound for the kinetic energy in the Haff's law \eqref{Haffs_Law}. One has
\begin{prop}
	\label{proposition_estimation_Hbf}
	For any non-negative measurable functions $f_1$, $f_2$ such that
	\begin{equation*}
		f_1, f_2 \in L^1(\R^d) \qquad \int_{\R^d} m_1\,f_1 = 1, \qquad \int_{\R^d} m_2\,f_2 = 1, \qquad \mathcal{E}(f) < \infty,
	\end{equation*} 
	with $\mathcal{H}(f) < \infty$, define
	\begin{equation}
		\mathbf{H}(f) :=  \int_{\R^d} f_1(v) |\log f_1(v)| \,dv + \int_{\R^d} f_2(v) |\log f_2(v)| \,dv.
	\end{equation}
	Then, there exists a constant $C>0$ such that
	\begin{equation}
		\label{inequality_mathbfH}
		\mathbf{H}(f) \leq \mathcal{H}(f) + C \left( \mathcal{E}(f)  \right)^{5/3}
	\end{equation}
	and another constant $K>0$ such that
	\begin{equation}
		\label{inequality_mathcalE}
		\mathcal{E}(f) \geq K \exp \left( - \frac{4}{3} \mathbf{H}(f) (m_1+m_2) \right) .
	\end{equation}
\end{prop}

\begin{proof}
	\par \textit{Inequality \eqref{inequality_mathbfH}.} We first prove \eqref{inequality_mathbfH}. Let us  consider the sets $\mathcal{K}_i = \{v \in \R^d \, \text{s.t.} \, f_i(v) < 1 \}$, thus
	\begin{align*}
		\mathbf{H}(f) &=  \int_{\R^d \setminus \mathcal{K}_1} f_1(v) \log f_1(v) \, dv +  \int_{\R^d \setminus \mathcal{K}_2} f_2(v) \log f_2(v) \, dv \\  & - \int_{\mathcal{K}_1} f_1(v) \log f_1(v) \, dv -  \int_{\mathcal{K}_2} f_2(v) \log f_2(v) \, dv \\
		&=  \mathcal{H}(f) + 2 \int_{\mathcal{K}_1} f_1(v) \log \left(\frac{1}{f_1(v)} \right) \, dv + 2  \int_{\mathcal{K}_2} f_2(v) \log \left(\frac{1}{f_2(v)} \right) \, dv.
	\end{align*}
	For any $\nu > 0$, let us define 
	\begin{align*}
		\mathcal{W}_i = \{v \in \R^d \, \text{s.t.} \, f_i(v) \geq \exp(- m_i^{k/2} \, \nu |v|^k \} \qquad i \in \{1, 2\}.
	\end{align*}
	If $v \in \mathcal{K}_i \cap \mathcal{W}_i$, then $\log (f_i^{-1}(v)) \leq m_i^{k/2} \, \nu |v|^k$, then
	\begin{align*}
		\mathbf{H}(f) &\leq \mathcal{H}(f) + 2 \nu \int_{\mathcal{K}_1 \cap \mathcal{W}_1} m_1^{k/2} f_1(v) |v|^k \, dv + 2 \int_{\mathcal{K}_1 \cap \mathcal{W}_1^c} f_1(v) \log \left(\frac{1}{f_1(v)} \right) \, dv \\
		&\phantom{\leq}+ 2 \nu \int_{\mathcal{K}_2 \cap \mathcal{W}_2} m_2^{k/2} f_2(v) |v|^k \, dv + 2 \int_{\mathcal{K}_2 \cap \mathcal{W}_2^c} f_2(v) \log \left(\frac{1}{f_2(v)} \right) \, dv \\
		&\leq \mathcal{H}(f) + 2 \nu \mu_k + \frac{4}{e} \left[ \int_{\mathcal{K}_1 \cap \mathcal{W}_1^c} \sqrt{f_1(v)} \, dv + \int_{\mathcal{K}_2 \cap \mathcal{W}_2^c} \sqrt{f_2(v)} \, dv \right] \\ 
		&\leq \mathcal{H}(f) + 2\nu \mu_k + \frac{4}{e} \int_{\R^d} \exp \left( -\nu (m_1^{k/2}+m_2^{k/2}) \frac{|v|^k}{2} \right) \, dv,
	\end{align*}
	where we have used that $x\log(1/x) \leq \frac{2}{e}\, \sqrt{x}$ and the expression of $\mu_k$ defined in \eqref{moments_expression_mu}. Computing the last term explicitly, we have
	\begin{equation*}
		\int_{\R^d} \exp \left( -\nu (m_1^{k/2}+m_2^{k/2}) \frac{|v|^k}{2} \right) \, dv = \frac{|\mathbb{S}^n|}{k} \left(\frac{2}{\nu(m_1^{k/2}+m_2^{k/2})} \right)^{\frac{n}{k}} \Gamma \left(\frac{n}{k} \right),
	\end{equation*}
	where $\Gamma(\cdot)$ is the Gamma function. Then, for any $k \in (0,n)$, there exists a positive constant $C_{n,k,m_1,m_2}$ such that
	\begin{equation*}
		\mathbf{H}(f) \leq \mathcal{H}(f) + 2\nu \mu_k + C_{n,k,m_1,m_2} \nu^{-\frac{n}{k}}.
	\end{equation*}
	Maximizing the parameter $\nu>0$ we obtain that
	\begin{equation*}
		\mathbf{H}(f) \leq \mathcal{H}(f) + C (\mathcal{E}(f))^{\frac{n+k}{n}},
	\end{equation*}
	where $C$ depends both on $n$, $k$ and the masses $m_1$, $m_2$ and we conclude the step choosing $n=3$ and $k=2$.\\
	\textit{Inequality \eqref{inequality_mathcalE}.} Now we have to prove that
	\begin{equation}
		\mathcal{E}(f) \geq K \exp \left( -\frac{4}{3} \mathbf{H}(f) (m_1+m_2)\right)
	\end{equation}
	We shall prove the inequality for a general moment of the form
	\begin{equation}
		\label{inequality_muk}
		\mu_k(f) \geq C(n,k,\eps) \exp \left(-\frac{k}{n(1-\eps)} \mathbf{H}(f) \right).
	\end{equation}
	We denote $B_R$ the ball with center in the origin and radius $R >0$, then we have
	\begin{equation}
		\label{inequality_momentum_entropy_energy}
		\begin{aligned}
			\mu_k(f) = \int_{B_R} (m_1^{k/2} f_1(v) + m_2^{k/2} f_2(v)) |v|^k \, dv + \int_{B_R^c} (m_1^{k/2} f_1(v) + m_2^{k/2} f_2(v)) |v|^k \, dv \\
			\geq R^k \left(1 - \int_{B_R} (m_1^{k/2} f_1(v) + m_2^{k/2} f_2(v)) dv \right)
		\end{aligned}
	\end{equation}
	We can control $\mu_k(f)$ from below if we find a control on the mass of $f$ on a suitable ball $B_R$. Using the generalized Young's inequality we obtain $\forall \lambda > 0$:
	\begin{align*}
		\int_{B_R} (m_1^{k/2} f_1(v) + m_2^{k/2} f_2(v)) \, dv \leq &\frac{m_1^{k/2}}{\lambda} \int_{B_R} f_1(v) \log f_1(v) + \frac{m_2^{k/2}}{\lambda} \int_{B_R} f_2(v) \log f_2(v) +\\ - &\frac{1}{\lambda}(\log \lambda + 1) \int_{B_R} (m_1^{k/2} f_1(v) + m_2^{k/2} f_2(v)) dv + \exp(\lambda) |B_R|. 
	\end{align*}
	Using the definition of the Boltzmann entropy for gas mixtures, we get
	\begin{equation}
		\frac{1}{\lambda}( \log \lambda + \lambda + 1) \int_{B_R} \int_{B_R} \left(m_1^{k/2} f_1(v) + m_2^{k/2} f_2(v) \right)\, dv \leq  \frac{m_1^{k/2}+m_2^{k/2}}{\lambda} \mathbf{H}(f) + \exp(\lambda) |B_R|,    
	\end{equation}
	which is equivalent to
	\begin{equation*}
		(\log(\lambda \exp(\lambda)) + 1) \int_{B_R} (m_1^{k/2} f_1(v) + m_2^{k/2} f_2(v)) \, dv \leq (m_1^{k/2}+m_2^{k/2})\mathbf{H}(f) + \lambda \exp(\lambda)|B_R|
	\end{equation*}
	Following the same strategy of \cite{alonso2013}, we set $x=\log(\lambda \exp(\lambda))$ and we assume $x>-1$. Then the last inequality could be rewritten as
	\begin{equation}
		\int_{B_R} (m_1^{k/2} f_1(v) + m_2^{k/2} f_2(v)) \, dv \leq \frac{(m_1^{k/2}+m_2^{k/2})\mathbf{H}(f) + \exp(x)|B_R|}{1+x} \qquad \forall\, x>-1.
	\end{equation}
	The minimal value of the right hand side is reached by $x_0$ such that
	\begin{equation*}
		x_0 \exp(x_0) |B_R| = (m_1^{k/2}+m_2^{k/2})\mathbf{H}(f).
	\end{equation*}
	Defining $W$ the inverse of the map $x \to x\exp(x)$, which corresponds to the Lambert function, then $x_0=W((m_1^{k/2}+m_2^{k/2})\mathbf{H}(f)/|B_R|)$ and
	\begin{multline*}
		\int_{B_R} (m_1^{k/2} f_1(v) + m_2^{k/2} f_2(v)) \leq x_0\,    \exp(x_0) |B_R| \\ = \exp\left(W \left(\frac{(m_1^{k/2}+m_2^{k/2})\mathbf{H}(f)}{|B_R|} \right) \right) |B_R|.
	\end{multline*}
	Combining this last inequality with \eqref{inequality_momentum_entropy_energy}, we obtain
	\begin{equation}
		\mu_k(f) \geq R^k \left(1- \exp\left(W \left(\frac{(m_1^{k/2}+m_2^{k/2})\mathbf{H}(f)}{|B_R|} \right) \right) |B_R|\right).
	\end{equation}
	We can now optimize the parameter $R$ and for any $\eps \in (0,1)$ we can find
	\begin{multline*}
		1- \exp\left(W \left(\frac{(m_1^{k/2}+m_2^{k/2})\mathbf{H}(f)}{|B_R|} \right) \right) |B_R|=\eps \\\iff |B_R| = (1-\eps) \exp\left(-\frac{(m_1^{k/2}+m_2^{k/2})\mathbf{H}(f)}{1-\eps} \right).
	\end{multline*}
	Thus, for any $\eps >0$, if we choose $R_0$ such that the previous equality is satisfied, then
	\begin{equation}
		\mu_k \geq \eps R_0^k.
	\end{equation}
	In particular, since $|B_{R_0}|=|\mathbb{S}^2|\,R_0^3$, then
	\begin{equation}
		\mathcal{E}(f) \geq \eps \left(\frac{3(1-\eps)}{4 \pi} \right)^{2/3} \exp \left(-\frac{2 \mathbf{H}(f)(m_1+m_2)}{3(1-\eps)} \right),
	\end{equation}
	and we conclude by choosing $\eps=1/2$.
\end{proof}
\begin{cor}
	Under the hypothesis of Proposition \ref{proposition_estimation_Hbf} we proved \eqref{inequality_muk}.
\end{cor}
Now we are ready to provide the main result of this Section, namely the generalized Haff's law for constant restitution coefficient.
\begin{thm}
	Let $e \in (0,1)$ be a constant restitution coefficient and let the initial distribution $f^{in}$ satisfy \ref{hypothesis_initial_moments} with $\mathcal{E}(f^{in}) < \infty$ and $\mathcal{H}(f^{in}) < \infty$. Let $f(t,v)$ be the unique solution to the multi-species Boltzmann equation \eqref{homogeneous_BEs} with this initial data. Then, there exist two positive constants $a, b >0$ such that
	\begin{equation}
		\label{estimation_sqrtE_thm}
		\int_0^t \sqrt{\mathcal{E}(f)}\,ds \geq \frac{1}{b} \log(1+a\,b\,t), \quad \forall\,t \geq 0.
	\end{equation}
	As a consequence,
	\begin{equation*}
		\sup \Big\{ \lambda \geq 0, \sup_{t \geq 0} (1+t)^{2 \lambda} \mathcal{E}(t) < \infty \Big\} = \inf \Big\{ \lambda > 0, \sup_{t \to \infty} (1+t)^{2 \lambda} \mathcal{E}(t) > 0 \Big\} = 1.
	\end{equation*}
\end{thm}

\begin{proof}
	Let us recall the evolution of the entropy for a constant restitution coefficient $e \in (0,1)$ for the hard-spheres model
	\begin{equation}
		\frac{d}{dt} \mathcal{H}(f(t)) = -\mathcal{D}_e(f(t)) + \frac{1-e^2}{e^2} \int_{\R^d \times \R^d} |u| \left[f_{1*}\,f_1 + 2 f_{2*}\,f_1 + f_{2*}\,f_2 \right] \, dv dv_*.
	\end{equation}
	In particular
	\begin{align*}
		\frac{d}{dt}\mathcal{H}(f(t)) \leq \, & \frac{1-e^2}{e^2} \int_{\R^d \times \R^d}  |u| [f_{1*} f_1 + 2 f_{2*}\,f_1 + f_{2*} f_2] \, dv \, dv_*,
	\end{align*}
	Following the strategy in \cite{alonso2013}, we can show that
	\begin{equation}
		\frac{d}{dt}\mathcal{H}(f(t)) \leq  2^{\gamma+1} \,  A \,\mu_{\frac{1}{2}}(t),
	\end{equation}
	where $\mu_k(t)$ are the k$^{\text{th}}$ order moments defined in \eqref{moments_expression_mu}. Indeed, one has
	\begin{equation*}
		\int_{\R^d \times \R^d} |u| (f_1 + f_2) (f_{1*} + f_{2*}) dv\,dv_* \leq A \int_{\R^d} m_1 |v| (m_1^{1/2} f_1 + m_2^{1/2} f_2) f_{1*}dv = 2 A  \mu_{\frac{1}{2}}(t),
	\end{equation*}
	where $A=1/\min\{m_1, m_2\}$. This implies that
	\begin{equation*}
		\frac{d}{dt} \mathcal{H}(f(t)) \leq 2 A \frac{1-e^2}{e^2} \mu_{\frac{1}{2}}(t) \leq 2 A \frac{1-e^2}{e^2} \sqrt{\mathcal{E}(t)}, \qquad \forall \, t \geq 0.
	\end{equation*}
	Integrating this inequality yields
	\begin{equation}
		\mathcal{H}(f(t)) \leq \mathcal{H}(f^{in}) + 2 A \frac{1-e^2}{e^2} \int_0^t \sqrt{\mathcal{E}(s)}\,ds \qquad \forall \, t \geq 0.
	\end{equation}
	Hence, there exists a constant $K_1 >0$ such that
	\begin{equation}
		\mathbf{H}(f(t)) \leq K_1 + 2 A \frac{1-e^2}{e^2} \int_0^t \sqrt{\mathcal{E}(s)}\,ds \qquad \forall \, t \geq 0.
	\end{equation}
	Therefore, from Proposition \ref{proposition_estimation_Hbf} there exists another constant $C>0$ such that
	\begin{equation}
		\label{inequality_E}
		\mathcal{E}(t) \geq C \exp \left(-\frac{4}{3} (m_1+m_2) K_1 - \frac{8\,(1-e^2)}{3\,e^2} (m_1+m_2) \int_0^t \sqrt{\mathcal{E}(s)}\,ds \right), \qquad \forall\,t \geq 0.
	\end{equation}
	Defining two new positive constants
	\begin{equation*}
		a := \sqrt{C} \exp \left( -\frac{2}{3} (m_1+m_2) K_1\right), \qquad b:= \left(\frac{8\,(1-e^2)}{3\,e^2} (m_1+m_2) \right), 
	\end{equation*}
	we can rewrite \eqref{inequality_E} as 
	\begin{equation*}
		\sqrt{\mathcal{E}(t)} \geq a \exp \left(-b\int_0^t \sqrt{\mathcal{E}(s)}\,ds \right), \qquad \forall\, t \geq 0,
	\end{equation*}
	and we conclude the first part of the proof.\\
	Now the aim is to prove that the only possible algebraic rate for the cooling of the temperature $\mathcal{E}(t)$ is $(1+t)^{-2}$. Let us set
	\begin{equation*}
		\Gamma_1 := \left\{ \lambda \geq 0 : \sup_{t\leq 0} (1+t)^{2\lambda} \mathcal{E}(t) < \infty \right\}.
	\end{equation*}
	Since \eqref{Haffs_Law_Hardspheres} holds, then $1 \in \Gamma_1$, which implies that $\Gamma_1$ is a non-empty set. Meanwhile, the inequality \eqref{estimation_sqrtE_thm} provides $\sup (\Gamma_1)=1$.\\
	Let us define
	\begin{equation*}
		\Gamma_2 := \left\{ \lambda \geq 0 : \limsup_{t \to \infty} (1+t)^{2\lambda} \mathcal{E}(t) >0 \right\}.
	\end{equation*} 
	The set $\Gamma_2$ is non-empty, since it holds
	\begin{equation}
		\label{inequality_E_Gamma2}
		\mathcal{E}(t) \geq r_2 (1+t)^{-2\lambda_0},
	\end{equation}
	for a constant $\lambda_0 >0$. This could be easily verified starting from \eqref{inequality_mathbfH} which, under the hypothesis of bounded $\mathcal{E}(t)$ implies
	\begin{equation*}
		\mathcal{H}(f(t)) \leq k_0 + K_0 \log(1+t), \qquad \forall t \geq 0,
	\end{equation*} 
	where $k_0, K_0 > 0$. Using again Proposition \ref{proposition_estimation_Hbf} there exists another constant $C>0$ such that
	\begin{equation*}
		\mathcal{E}(t) \geq C \exp \left(-\frac{4}{3} (m_1+m_2) \mathbf{H}(f(t)) \right) \geq c \left(-\frac{4 K_0}{3} (m_1+m_2) \log (1+t) \right),
	\end{equation*}
	where $c= C \exp(-4 k_0 (m_1+m_2) /3)$. Therefore for the choice $\lambda_0=\frac{2 K_0}{3} (m_1+m_2)$, we get \eqref{inequality_E_Gamma2}.\\
	Now, we would like to show that $\inf (\Gamma_2)=1$ by contradiction. Let us suppose that $\inf (\Gamma_2) = \bar{\lambda}$, then $\bar{\lambda} > 1$ for \eqref{Haffs_Law_Hardspheres}. This implies that for every $\lambda \in (1,\bar{\lambda})$ $\lambda \notin \Gamma_2$ and consequently $\lambda \in \Gamma_1$. This is impossible since $\sup (\Gamma_1) = 1$ and thus we conclude by contradiction.
\end{proof}

\begin{rem}
	In the case of constant restitution coefficients, we cannot give a precise estimation on the constant in the generalized Haff's law, but we are able to proe that $(1+t)^{-2}$ is the only possible algebraic rate for $\mathcal{E}(t)$.\\
	For a non constant restitution coefficient in a weakly inelastic regime, it is possible to prove a stronger result. Indeed, the generalized Haff's law for a restitution coefficient which satisfies $1-e(|u|) \simeq \alpha |u|^\gamma$ follows from Proposition \ref{proposition_estimation_Hbf} and \cite[Theorem $3.7$]{alonso2010}. The only difference in the proof between the single-species case and the multi-species one, relies on the use of conservation properties, which now hold for the total moments and not for the partial ones.
\end{rem}

    \section*{Acknowledgement}
    TR and TT received funding from the European Union's Horizon Europe research and innovation program under the Marie Sklodowska-Curie Doctoral Network DataHyKing (Grant No. 101072546), and are supported by the French government, through the UniCA$_{JEDI}$ Investments in the Future project managed by the National Research Agency (ANR) with the reference number ANR-15-IDEX-01.

\appendix

\label{Appendix_Orlicz}
\section{Orlicz Spaces Framework}
\stoptoc
In this appendix we detail the functional framework used for the results in Section \ref{Section_MischlerMouhot}, providing some results about Orlicz spaces. These results can be found in \cite{raoren1991} or in the Appendix of \cite{mischlermouhot2006}.\\
Let $\Lambda: \R_+ \to \R_+$ be a function $C^2$ strictly increasing, convex and such that
\begin{equation}
	\label{Appendix_Assumption_Orlicz1}
    \Lambda(0)=\Lambda'(0)=0,
\end{equation}
\begin{equation}
	\label{Appendix_Assumption_Orlicz2}
    \Lambda(2t) \leq c_\Lambda \Lambda(t), \quad \forall \, t \geq 0,
\end{equation}
for a positive constant $c_\Lambda$, and which is superlinear, namely
\begin{equation}
	\label{Appendix_Assumption_Orlicz3}
    \frac{\Lambda(t)}{t} \to_{t \to +\infty} + \infty.
\end{equation}
We can define $L^\Lambda$ the set of measurable functions $f: \R^N \to \R$ such that
\begin{equation}
    \int_{\R^N} \Lambda(|f(v)|) dv < + \infty.
\end{equation}
This is a Banach space for the norm
\begin{equation}
    ||f||_{L^\Lambda} = \inf \left\{ \lambda > 0 \text{ s. t. } \int_{\R^N} \Lambda \left( \frac{|f(v)|}{\lambda} \right) dv \leq 1 \right\},
\end{equation}
which is called \textit{Orlicz space} associated to the function $\Lambda$.\\
As underlined in \cite{mischlermouhot2006}, a refined version of the De la Vallée-Poussin Theorem ensures that for any $f \in L^1(\R^N)$ there exists a function $\Lambda$ satisfying all the properties above and such that
\begin{equation*}
    \int_{\R^N} \Lambda(|f(v)|)\,dv < + \infty
\end{equation*}
For the sake of completeness, we recall the statement of the De la Vallée-Poussin Theorem \cite[Theorem 24-II]{dellacherie1978}.
\begin{thm}[De la Vallée-Poussin]
	\label{de_la_vallee_poussin}
	Let $\mathcal{H}$ be a subset of $L^1$. The following properties are equivalent:
	\begin{enumerate}
		\item $\mathcal{H}$ is uniformly integrable.
		\item There exists a positive function $\Theta$ defined on $\R_+$ such that $\Theta$ is superlinear and
		\begin{equation}
			\sup_{f \in \mathcal{H}} \mathbb{E}[G \circ |f|] < \infty.
		\end{equation}
	\end{enumerate}
\end{thm} 
Moreover, we recall the Dunford-Pettis compactness criterion \cite[Theorem 27-II]{dellacherie1978}, which is fundamental for the proof of the existence of solutions to the multi-species inelastic Boltzmann equation.
\begin{thm}[Dunford-Pettis compactness criterion]
	\label{dunford_pettis}
	Let $\mathcal{H}$ be a subset of the space $L^1$. The following three properties are equivalent:
	\begin{enumerate}
		\item $\mathcal{H}$ is uniformly integrable.
		\item $\mathcal{H}$ is relatively compact in $L^1$ with the weak topology $\sigma(L^1, L^\infty)$.
		\item Every sequence of elements of $\mathcal{H}$ contains a subsequence which converges in the sense of the topology $\sigma(L^1, L^\infty)$.
	\end{enumerate}
\end{thm}

\subsection{H\"older's inequality in Orlicz spaces}
Let $\Lambda$ be a $C^2$ function, strictly increasing, convex satisfying \eqref{Appendix_Assumption_Orlicz1}-\eqref{Appendix_Assumption_Orlicz2}-\eqref{Appendix_Assumption_Orlicz3}. Let us define $\Lambda^*$ its complementary Young function, namely
\begin{equation*}
	\Lambda^*(y) = y\,\left(\Lambda'\right)^{-1}(y) - \Lambda \left(\left( \Lambda' \right)^{-1} (y)\right) \qquad \forall \, y \geq 0.
\end{equation*}
The Young's inequality holds
\begin{equation*}
	xy \leq \Lambda(x) + \Lambda^*(y), \qquad \forall\, x,\,y \geq 0.
\end{equation*}
Hence, we can define the following norm on the Orlicz space $L^{\Lambda^*}$:
\begin{equation}
	\label{norm_N^*}
	N^{\Lambda^*} (f) = \sup \left\{ \int_{\R^N} |f\,g|\,dv \text{ s.t } \int_{\R^N} \Lambda \left(|g(v)| \right) dv \leq 1 \right\}.
\end{equation}
We can give the following result, essential tool for the estimates of the collision operator in Orlicz space.
\begin{thm}[\cite{raoren1991}, Chapter III]~ 
	\begin{enumerate}[(i)]
		\item The following H\"older's inequality holds for any $f \in L^\Lambda$, $g \in L^{\Lambda^*}$:
		\begin{equation*}
			 \int_{\R^N} |f\,g|\,dv \leq ||f||_{L^\Lambda} N^{\Lambda^*} (g).
		\end{equation*}
		\item The equality holds if and only if there is a constant $k_* \in (0,+\infty)$ such that
		\begin{equation*}
			\left(\frac{|f|}{||f||_{L^\Lambda}} \right) \left(\frac{k_*\,|g|}{N^{\Lambda^*}(g)} \right) = \Lambda \left(\frac{|f|}{||f||_{L^\Lambda}} \right) + \Lambda^* \left(\frac{k_*\,|g|}{N^{\Lambda^*}(g)} \right),
		\end{equation*}
		for a.e. $v \in \R^N$.
	\end{enumerate}
\end{thm}

\subsection{Differential in Orlicz norms}
The last technical result we need to propagate bounds on Orlicz norms along the flow of the Boltzmann equation is the following.
\begin{thm}
	\label{theorem_differentiation_orlicz}
	Let $\Lambda$ be a $C^2$ function, strictly increasing, convex satisfying \eqref{Appendix_Assumption_Orlicz1}-\eqref{Appendix_Assumption_Orlicz2}-\eqref{Appendix_Assumption_Orlicz3} and let $f \in C^1([0,T], L^\Lambda)$ be a positive function such that $f \not\equiv 0$. Then we have
	\begin{equation*}
		\frac{d}{dt} ||f||_{L^\Lambda} = \left[N^{\Lambda^*} \left( \Lambda' \left( \frac{|f|}{||f||_{L^\Lambda}} \right) \right) \right]^{-1} \int_{\R^N} \partial_t f \, \Lambda' \left( \frac{|f|}{||f||_{L^\Lambda}} \right) \,dv. 
	\end{equation*}
\end{thm}
\resumetoc

\bibliography{References}
\bibliographystyle{acm}

\end{document}